\documentclass[12pt]{article}

\usepackage[english]{babel}
\usepackage{amssymb}
\usepackage{amsfonts,amsmath,mathtools}
\usepackage{graphics}
\usepackage{lipsum}
\usepackage[ruled]{algorithm2e}
\usepackage{float}
\usepackage{url}
\usepackage{subfig}
\usepackage{xcolor}
\usepackage[export]{adjustbox}
\usepackage{tikz-cd}
\usepackage{titlesec}

\usepackage{lipsum} 

\titleformat{\section}
{\centering\normalfont\scshape}{\thesection.}{0.5em}{}
\titleformat{\subsection}
{\normalfont}{\thesubsection.}{0.5em}{\textbf}

\usepackage{etoolbox}
\apptocmd{\thebibliography}{\setlength{\itemsep}{-3pt}}{}{}

\usepackage{bm}
\usepackage{enumitem}
\usepackage{faktor}
\usepackage{comment}

\usepackage[top=4cm,bottom=3.5cm,left=3.1cm,right=3.1cm,asymmetric]{geometry}

\newcommand\blfootnote[1]{%
	\begingroup
	\renewcommand\thefootnote{}\footnote{#1}%
	\addtocounter{footnote}{-1}%
	\endgroup
}

\newtheorem{Lem}{Lemma}[section]
\newtheorem{Prop}[Lem]{Proposition}
\newtheorem{Cor}[Lem]{Corollary}
\newtheorem{Thm}[Lem]{Theorem}

\newtheorem{Def}[Lem]{Definition}
\newtheorem{Rem}[Lem]{Remark}

\newtheorem{Expl}[Lem]{Example}

\newtheorem{Conj}[Lem]{Conjecture}
\newtheorem{Que}[Lem]{Question}

\newenvironment{proof}[1][Proof]{\textrm{\em #1.} }{\hfill$\Box$\medskip\medskip}

\newcommand\Min{{\operatorname{Min}}}

\newcommand\supp{{\operatorname{supp}}}
\newcommand\cosupp{{\operatorname{cosupp}}}

\newcommand\lex{{\operatorname{lex}}}
\newcommand\slex{{\operatorname{slex}}}

\newcommand\pd{{\operatorname{pd}}}

\newcommand\depth{{\operatorname{depth}}}

\newcommand\height{{\operatorname{height}}}

\def\L{\mathcal{L}}

\def\NZQ{\mathbb}
\def\NN{{\NZQ N}}

\let\emptyset\varnothing
\let\epsilon\varepsilon

\def\alt{\textup{height}}

\def\p{\mathfrak{p}}
\def\q{\mathfrak{q}}

\begin{document}

\title{\bf\normalsize\MakeUppercase{Classification of Cohen--Macaulay} $t$--\MakeUppercase{spread lexsegment ideals via simplicial complexes}}

\author{Marilena Crupi, Antonino Ficarra}	

\newcommand{\Addresses}{{
 \footnotesize
 \textsc{Department of Mathematics and Computer Sciences, Physics and Earth Sciences, University of Messina, Viale Ferdinando Stagno d'Alcontres 31, 98166 Messina, Italy}
\begin{center}
 \textit{E-mail addresses}: \texttt{mcrupi@unime.it}; \texttt{antficarra@unime.it}
\end{center}

}}
\date{}
\maketitle
\Addresses

\begin{abstract}
We study the minimal primary decomposition of completely $t$--spread lexsegment ideals via simplicial complexes. 
We determine some algebraic invariants of such a class of $t$--spread ideals. Hence, we classify all $t$--spread lexsegment ideals which are Cohen--Macaulay.
\blfootnote{
	\hspace{-0,3cm} \emph{Keywords:} Betti numbers, $t$--spread ideals, primary decomposition, simplicial complexes, Cohen--Macaulay rings.
	}
\end{abstract}

\section*{Introduction}
In combinatorial commutative algebra, one often associates to a combinatorial object $X$ a suitable monomial ideal $I_X$ in a polynomial ring $S=K[x_1,\dots,x_n]$ in finitely many variables with coefficients in a field $K$. Distinguished combinatorial properties of the object $X$ are reflected by algebraic properties of the ideal $I_X$, and conversely. One says that $X$ is Cohen--Macaulay if the ring $S/I_X$ is Cohen--Macaulay. A typical task is to classify all Cohen--Macaulay objects $X$ of a given class. For instance, in combinatorics, to a simplicial complex $\Delta$ on the vertex set $[n]$ one associates the Stanley--Reisner ideal $I_\Delta$, \emph{i.e.}, the ideal generated by all squarefree monomials $x_{i_1}\cdots x_{i_r}$, $r\le n$, such that $\{i_1, \ldots, i_r\}\notin \Delta$. 
A simplicial complex $\Delta$ is said Cohen--Macaulay if the ideal $I_\Delta$ is Cohen--Macaulay, \emph{i.e.}, $S/I_\Delta$ is a Cohen--Macaulay ring. 
To classify all Cohen--Macaulay simplicial complexes is an hopeless task. However, many results can be found, for instance, in \cite{JT}.

The aim of this article is to classify all Cohen--Macaulay $t$--spread lexsegment ideals for $t\ge1$. 

Let $S=K[x_1,\dots,x_n]$ be the polynomial ring in $n$ variables with coefficients over a fixed field $K$. 
Let $t\ge 0$ be an integer. A monomial $u=x_{i_1}x_{i_2}\cdots x_{i_d}$ is \textit{$t$--spread} if $i_{j+1}-i_j\ge t$, for all $j=1,\dots,d-1$. 
If $t\ge1$, a $t$--spread monomial is a squarefree monomial. 
A monomial ideal is $t$--spread if it is generated by $t$--spread monomials. Such a topic introduced in \cite{EHQ} has immediately captured the attention of many algebraists (see, for instance, \cite{LA, AFC1, AFC2, ACF3, AEL, CAC, RD, FC1}).
Let $t\ge1$. A $t$--spread ideal $I$ of $S$ is said a \emph{$t$--spread lexsegment ideal} if for all $t$--spread monomials $u\in I$ and all $t$--spread monomials $v\in S$ with $\deg(u)=\deg(v)$ and such that $v\ge_{\slex}u$, then $v\in I$ \cite{CAC}, where $\ge_{\slex}$ is the squarefree lexicographic order \cite{AHH2, JT}. Therefore, for $t=1$, such a definition coincides with the \emph{classical} definition of squarefree lexsegment ideal \cite{AHH2}. 
Let $t\ge 1$ and let $M_{n,d,t}$ be the set of all $t$--spread monomials of $S$ of degree $d$. Assume that $M_{n,d,t}$ is endowed with the squarefree lexicographic order $\ge_{\slex}$.  Let $u,v\in M_{n,d,t}$. If $u\ge_{\slex}v$, the set $\mathcal{L}_t(u,v)=\{w\in M_{n,d,t}:u\ge_{\slex}w\ge_{\slex}v\}$ is called an \emph{arbitrary $t$--spread lexsegment}.
If $u$ is the maximum monomial of $M_{n,d,t}$, or $v$ is the minimum monomial of $M_{n,d,t}$, we say that $\mathcal{L}_t(u,v)$ is an \emph{initial} or a \emph{final $t$--spread lexsegment}, respectively. 
The ideal generated by an arbitrary $t$--spread lexsegment is called \emph{arbitrary $t$--spread lexsegment ideal}. It is clear that the notion of initial $t$--spread lexsegment ideal coincides with the notion of $t$--spread lexsegment ideal in \cite{CAC}. 
An arbitrary $t$--spread lexsegment ideal $I = (\mathcal{L}_t(u,v))$ is said \emph{completely $t$--spread lexsegment} if $I=(\mathcal{L}_t(\max(M_{n,d,t}),v))\cap(\mathcal{L}_t(u,\min(M_{n,d,t})))$, \cite{FC1}. 
One can easily observe that every initial and every final $t$--spread lexsegment ideal is a completely $t$--spread lexsegment, but not all arbitrary $t$--spread lexsegment ideals are completely $t$--spread lexsegment ideals. In \cite{FC1}, the authors characterized all completely $t$--spread lexsegment ideals and classified all completely $t$--spread lexsegment ideals with a linear resolution. The case $t=0$ has been faced by De Negri and Herzog in \cite{DH}.

In this article we determine the minimal primary decomposition of completely $t$--spread lexsegment ideals and some algebraic invariants of such a class of ideals by tools from the simplicial complex theory.  
Hence, we classify all $t$--spread lexsegment ideals which are Cohen--Macaulay. The case $t=0$ has been considered in \cite{EOS2010} and the case $t=1$ has been analyzed in \cite{BST}.

The article is organized as follows. Section \ref{sec1} contains preliminary notions and results.
In Section \ref{sec2}, we explicitly compute the minimal primary decomposition for
initial and final $t$--spread lexsegment ideals (Theorems \ref{primdecompinitialtspreadlex}, \ref{primdecompfinaltspreadlex}). As a consequence we are able to compute the standard primary decomposition for the class of completely $t$--spread lexsegments (Theorem \ref{primdecompgeneraltspreadlex}). Moreover, we obtain formulas for
the Krull dimension, the projective dimension and the depth of $S/I$ when $I$ is an initial or a final $t$--spread lexsegment ideal.
In Section \ref{sec3}, we classify all $t$--spread lexsegment ideals which are Cohen--Macaulay by the results stated in Section \ref{sec2}. The notion of Betti splitting will be a crucial tool. Finally, Section \ref{sec4} contains our conclusions and perspectives.

All the examples are constructed by means of \emph{Macaulay2} packages one of which developed in \cite{LA}. Furthermore, some functions described in \cite{LA} have allowed us to test classes of Cohen--Macaulay $t$--spread lexsegment ideals whose determination has been fundamental for the development of Section \ref{sec3}.

\section{Preliminaries}\label{sec1}

Let $S=K[x_1,\dots,x_n]$ be the polynomial ring in $n$ indeterminates with coefficients in a field $K$. $S$ is a $\NN$--graded $K$--algebra with $\deg(x_i)=1$ for all $i$. 

If $I$ is a monomial ideal of $S$, we denote by $G(I)$ the unique minimal set of monomial generators of $I$. Moreover, we set $G(I)_j=\{u\in G(I):\deg(u)=j\}$.

Given a non empty subset $A\subseteq[n]$, we set ${\bf x}_A=\prod_{i\in A}x_i$. 
Moreover, we set ${\bf x}_{\emptyset}=1$. Every monomial $w\in S$ can be written as $w=x_{i_1}x_{i_2}\cdots x_{i_d}$, with sorted indexes $1\le i_1\le i_2\le\dots\le i_d\le n$. We define the \textit{support} of $w$ as follows:
$$
\supp(w)=\{j: x_j \,\mbox{divides $w$} \}= \{i_1,i_2,\dots,i_d\}.
$$

With the notation above,  $\min(w)=\min\big\{i:i\in\supp(u)\big\}=i_1$, $\max(w)=\max\big\{i:i\in\supp(u)\big\}= i_d$. Moreover, we set $\max(1)=\min(1)=0$.
If $i_1\ne i_2\ne\dots\ne i_d$ we say that $u$ is \textit{squarefree}. In such a case, $u$ can be written as $u={\bf x}_A$ with $A=\supp(u)$.\\

We quote next definition from \cite{EHQ}.
\begin{Def}
	\rm Given $n\ge1,t\ge0$ and $u=x_{i_1}x_{i_2}\cdots x_{i_d}\in S$, 
	with $1\le i_1\le i_2\le\ldots\le i_d\le n$, we say that $u$ is \textit{$t$--spread} if $i_{j+1}-i_j\ge t$, for all $j=1,\dots,d-1$. We say that a monomial ideal $I$ of $S$ is \textit{$t$--spread} if it is generated by $t$--spread monomials.
\end{Def}

We note that any monomial ideal of $S$ is a $0$--spread monomial ideal, and any squarefree monomial ideal of $S$ is a $1$--spread monomial ideal.\\

Let $n,d,t\ge 1$.  We denote by $M_{n,d,t}$ the set of all $t$--spread monomials of $S$ of degree $d$. Throughout this paper we tacitly assume that $n\ge 1+(d-1)t$, otherwise $M_{n,d,t}=\emptyset$. From \cite{EHQ}, we have that $|M_{n,d,t}|=\binom{n-(t-1)(d-1)}{d}$. The monomial ideal generated by $M_{n,d,t}$ is denoted by $I_{n,d,t}$ and is called the \textit{$t$--spread Veronese ideal of degree $d$} of the polynomial ring $S=K[x_1,\dots,x_n]$ \cite{EHQ}.\\

From now on we consider $t\ge 1$ and assume that $M_{n,d,t}$ is endowed with the \emph{squarefree lexicographic order}, $\ge_{\slex}$, \cite{JT}. Recall that given $u=x_{i_1}x_{i_2}\cdots x_{i_d}, v=x_{j_1}x_{j_2}\cdots x_{j_d} \in M_{n,d,t}$, with $1\le i_1<i_2<\dots< i_d\le n$, $1\le j_1<j_2<\dots<j_d\le n$, then $u>_{\slex}v$ if
$$i_1=j_1,\ \dots,\ i_{s-1}=j_{s-1}\ \ \text{and}\ \ i_s<j_s$$
for some $1\le s\le d$. 

Let $U$ be a not empty subset of $M_{n,d,t}$. We denote by $\max(U)$ ($\min(U)$, respectively) the maximum (minimum, respectively) monomial $w\in U$, with respect to $\ge_{\slex}$. One has 
\begin{align*}
\max(M_{n,d,t})& = x_1x_{1+t}x_{1+2t}\cdots x_{1+(d-1)t}, \\
\min(M_{n,d,t})& = x_{n-(d-1)t}x_{n-(d-2)t}\cdots x_{n-t}x_n.
\end{align*}

In \cite{FC1}, the following definitions have been introduced.
\begin{Def}\rm
	Let $u,v\in M_{n,d,t}$, $u\ge_{\slex}v$. 
	The set
	\[\mathcal{L}_t(u,v) = \big\{w\in M_{n,d,t}:u\ge_{\slex}w\ge_{\slex}v\big\}\]
	is called an \emph{arbitrary $t$--spread lexsegment set}, or simply a \emph{$t$--spread lexsegment set}. The set
	\[\mathcal{L}_t^{i}(v) = \big\{w\in M_{n,d,t}:w\ge_{\slex}v\big\}=\mathcal{L}_t(\max(M_{n,d,t}),v)\]
	is called an \emph{initial $t$--spread lexsegment} and the set
	\[\mathcal{L}_t^{f}(u)=\big\{w\in M_{n,d,t}:w\le_{\slex}u\big\}=\mathcal{L}_t(u,\min(M_{n,d,t}))
	\]
	is called a \emph{final $t$--spread lexsegment}.
\end{Def}

One can observe that, if $u,v\in M_{n,d,t}$, with $u\ge_{\slex}v$, then $\mathcal{L}_t(u,v)=\mathcal{L}_t^i(v)\cap\mathcal{L}_t^f(u)$.

\begin{Def}\label{def:comp}\rm
	A $t$--spread monomial ideal $I$ of $S$ is called an \textit{arbitrary $t$--spread lexsegment ideal}, or simply a \textit{$t$--spread lexsegment ideal} if it is generated by an (arbitrary) $t$--spread lexsegment.\\ A $t$--spread lexsegment ideal $I=(\mathcal{L}_t(u,v))$ is a \textit{completely $t$--spread lexsegment ideal} if
	$$
	I=J\cap T,
	$$
	where $J=(\mathcal{L}_t^i(v))$ and $T=(\mathcal{L}_t^f(u))$.
\end{Def}

In \cite{FC1}, all completely $t$--spread lexsegment ideals have been characterized.\medskip

Now, we recall some important notions that will be useful in the sequel.\\

A $t$--spread monomial ideal $I$ is a \textit{$t$--spread strongly stable ideal} if for all $t$--spread monomial $u\in I$, all $j\in \supp(u)$ and all $1\le i< j$ such that $x_i(u/x_j)$ is $t$--spread, it follows that $x_i(u/x_j)\in I$.

Every initial $t$--spread lexsegment ideal $I$ of $S$ is a $t$--spread strongly stable ideal. Hence, in order to compute the graded Betti numbers of $I$ one may use the Ene, Herzog, Qureshi's formula (\cite[Corollary 1.12]{EHQ}):
\begin{equation}
\label{eq1}
\beta_{i,i+j}(I)\ =\ \sum_{u\in G(I)_j}\binom{\max(u)-t(j-1)-1}{i}.
\end{equation}

A $t$--spread final lexsegment ideal $I$ of $S$ is also a $t$--spread strongly stable ideal, but with the order of the variables reversed, \emph{i.e.}, $x_n> x_{n-1} > \cdots >x_1$. Hence, in order to compute their graded Betti numbers one may use the following \emph{modified} Ene, Herzog, Qureshi's formula:
\begin{equation}\label{Bettinumbersreversetspread}
\beta_{i,i+j}(I)\ =\ \sum_{u\in G(I)_j}\binom{n-\min(u)-t(j-1)}{i}.
\end{equation}

Let $>_{\lex}$ be the usual \textit{lexicographic order} on $S$, with $x_1>x_2>\cdots>x_n$, \cite{JT}. The next theorem collects some results from  \cite{FC1} that will be pivotal for the aim of this article.
\begin{Thm}\textup{\cite{FC1}}.\label{thm:compltspreadlex}
	Given $n,d,t\ge1$, let $u=x_{i_1}x_{i_2}\cdots x_{i_d}$ and $v=x_{j_1}x_{j_2}\cdots x_{j_d}$ be $t$--spread monomials of degree $d$ of $S$ such that $u\ge_{\slex}v$. 

	\begin{enumerate}\item[\em(1)] Let $I=(\L_t(u,v))$. The following conditions are equivalent:
	\begin{enumerate}
		\item[\em(a)] $I$ is a completely $t$--spread lexsegment ideal, \emph{i.e.}, $I=J\cap T$;
		\item[\em(b)] for every $w\in M_{n,d,t}$ with $w<_{\slex}v$ there exists an integer $s>i_1$ such that $x_s$ divides $w$ and $x_{i_1}(w/x_{s})\le_{\lex}u$. 
	\end{enumerate}
	
	\item[\em(2)] Let $I=(\L_t(u,v))$ be a completely $t$--spread lexsegment ideal with $\min(v)>\min(u)=1$. $I$ has a linear resolution if and only if one of the following conditions hold:
	\begin{enumerate}
		\item[\em(i)] $i_2=1+t$;
		\item[\em(ii)] $i_2>1+t$ and for the largest $w\in M_{n,d,t}$, $w<_{\lex}v$, we have $x_1(w/x_{\max(w)})\le_{\lex}x_1x_{i_2-t}x_{i_3-t}\cdots x_{i_d-t}$.
	\end{enumerate}
	\item[\em(3)] Let $I=(\L_t(u,v))$ be a completely $t$--spread lexsegment ideal and suppose $I=(\L_t(u,v))$ has a linear resolution. Then, for all $i\ge0$,
	$$
	\beta_{i}(I)=\sum_{w\in\mathcal{L}_t^f(u)}\binom{n-\min(w)-(d-1)t}{i}-\sum_{\substack{w\in\mathcal{L}_t^f(v)\\ w\ne v}}\binom{\max(w)-(d-1)t-1}{i}.
	$$
	\end{enumerate}
\end{Thm}

We close the section recalling some notions about simplicial complexes.

Given $n\in\mathbb{N}$, we set $[n]=\{1,2,\dots,n\}$. We recall that
	a \textit{simplicial complex} on the \textit{vertex set} $[n]$ is a family of subsets of $[n]$ such that 
	\begin{enumerate}
		\item[-] $\{i\}\in\Delta$ for all $i\in[n]$, and 
		\item[-] if $F\subseteq[n]$, $G\subseteq F$, we have $G\in\Delta$. 
	\end{enumerate}
	
	The dimension of $\Delta$ is the number $d=\max\{|F|-1:F\in\Delta\}$. Any $F\in\Delta$ is called a \textit{face} and $|F|-1$ is the \textit{dimension} of $F$. A \textit{facet} of $\Delta$ is a maximal face with respect to the inclusion. The set of facets of $\Delta$ is denoted by $\mathcal{F}(\Delta)$.

It is well known that for any squarefree ideal ($1$--spread ideal) $I$ of $S$ there exists a unique simplicial complex $\Delta$ on $[n]$ such that $I=I_\Delta$, where
$$
I_\Delta=({\bf x}_F:F\subseteq[n],F\notin\Delta)
$$
is the Stanley--Reisner ideal of $\Delta$ \cite{JT}. In the sequel when we say that a simplicial complex $\Delta$ is associated to a squarefree ideal $I$ of $S$, we mean that $I_{\Delta}= I$.

A simplicial complex $\Delta$ on $[n]$ is called \textit{pure} if its facets have the same dimension. We say that $\Delta$ is \textit{Cohen--Macaulay} if $K[\Delta]=S/I_{\Delta}=K[x_1,\dots,x_n]/I_\Delta$ is a Cohen--Macaulay ring. If $\Delta$ is Cohen--Macaulay then $\Delta$ is pure (see, for instance, \cite{JT}).

It is well known that any monomial ideal (not necessarily squarefree) has a unique minimal primary decomposition. We refer to it as the \textit{standard primary decomposition}. Moreover, for a squarefree monomial ideal, its standard primary decomposition is $I=\bigcap_{\p\in\textup{Min}(I)}\p$, \cite[Corollary 1.3.6]{JT}, where $\text{Min}(I)$ is the set of minimal primes of $I$. Every minimal prime is a monomial prime ideal, \emph{i.e.}, $\p=(x_{i_1},x_{i_2},\dots,x_{i_s})$, for some $A=\{i_1<i_2<\dots<i_s\}\subseteq[n]$. To denote such an ideal we also use the notation $\p_A$. 
\begin{Prop}\label{prop:AlexDuality1}
	\textup{\cite[Lemma 1.5.4]{JT}} The standard primary decomposition of $I_\Delta$ is
	$$
	I_{\Delta}=\bigcap_{F\in\mathcal{F}(\Delta)}\p_{[n]\setminus F}.
	$$
\end{Prop}

Given two positive integers $j\ge k$, we set $[j,k]=\{\ell\in\mathbb{N}:j\le\ell\le k\}$ and we associate to a $t$--spread monomial two special sets of integers.
	
	Let $t\ge1$ and let $w=x_{\ell_1}x_{\ell_2}\cdots x_{\ell_d}$ be a $t$--spread monomial of $S$ of degree $d$. 
	
	If $\max(w)\le n+1-t$, we define the \textit{$t$--spread support} of $w$ as follows:
\begin{equation}
	\label{eq:suppt}
	\supp_t(w)=\bigcup_{r=1}^d[\ell_r,\ell_r+(t-1)];
	\end{equation}
	whereas, if $\min(w)\ge t$, we define the \textit{$t$--spread cosupport} of $w$ as follows:
	\begin{equation}
	\label{eq:cosuppt}\cosupp_t(w)=\bigcup_{r=1}^{d}[\ell_r-(t-1),\ell_r].
	\end{equation}
	
	Moreover, we set $\supp_t(1)=\cosupp_t(1)=\emptyset$.

\begin{Rem}\em Let $w=x_2x_5x_{10}x_{13}\in S=K[x_1, \ldots, x_{13}]$ be a 3--spread monomial. One can note that $\max(w)=13 >13+1-3=11$. Furthermore, $\supp_3(w/x_{\max(w)})=\{2,3,4,5,6,7,10,11,12\}$ and $\cosupp_3(w/x_{\min(w)})=\{3,4,5,8,9,10,11,12,13\}$. 
	
Observe that for all $w\in M_{n,d,t}$, we can always compute $\supp_t(w/x_{\max(w)})$ and $\cosupp_t(w/x_{\min(w)})$.
\end{Rem}

In \cite[Theorem 2.3]{EHQ}, Ene, Herzog, Qureshi computed the standard primary decomposition of the $t$--spread Veronese ideal of degree $d$, $I_{n,d,t}$,
\begin{equation}\label{eq:decVeronese}
I_{n,d,t}=\bigcap_{D\in\mathcal{D}}\p_{[n]\setminus D},
\end{equation}
where
$$
\mathcal{D}=\Big\{\bigcup_{r=1}^{d-1}[\ell_r,\ell_r+(t-1)]\subseteq[n]\ :\ \ell_{r+1}-\ell_{r}\ge t,\ r=1,\dots,d-2\Big\}.
$$
Using our notation, each set $D\in\mathcal{D}$ may be written as $\supp_t(x_{\ell_1}\cdots x_{\ell_{d-1}})$, with 
$x_{\ell_1}\cdots x_{\ell_{d-1}}\in M_{n+1-t,d-1,t}$. Hence, the  primary decomposition (\ref{eq:decVeronese}) becomes
\begin{equation}\label{eq:primdecompIndt}
I_{n,d,t}=\bigcap_{w\in M_{n+1-t,d-1,t}}\p_{[n]\setminus\supp_t(w)}.
\end{equation}
In particular, $I_{n,d,t}$ is an \textit{height--unmixed} ideal, \emph{i.e.}, all its minimal primes have the same height, $n-(d-1)t$ \cite[Theorem 2.3]{EHQ}.
Furthermore, $I_{n,d,t}$ is a completely $t$--spread lexsegment ideal \cite{FC1}.\\

\section{The primary decomposition of\\ completely $t$--spread lexsegment ideals}\label{sec2}

In this section we study the primary decomposition of a completely $t$--spread lexsegment ideal of $S=K[x_1, \ldots, x_n]$. The case $t=1$ has been analyzed in \cite{OO}. The sets defined in (\ref{eq:suppt}) and (\ref{eq:cosuppt}) will be pivotal for our aim.\\

We start the section with some comments and remarks.

Let $n,d,t\ge 1$, $u=x_{i_1}x_{i_2}\cdots x_{i_d}$ and 
$v=x_{j_1}x_{j_2}\cdots x_{j_d}$ monomials of $M_{n,d,t}$, with $u\ge_{\slex}v$. 
Set $L=\mathcal{L}_t(u,v)$, $I=(L)$. We assume that $I$ is a completely $t$--spread lexsegment ideal, \emph{i.e.}, $I\ =\ J\cap T$,
with $J=(\mathcal{L}_t^i(v))$ and $T=(\mathcal{L}_t^f(u))$.

\begin{enumerate}
	\item[-] If $u=v$, $I=(L)=(u)$ is a principal ideal, and its standard primary decomposition is $I=(x_{i_1})\cap\dots\cap(x_{i_d})$. Therefore, we may assume $u>_{\slex}v$.
	\item[-] If $\deg(u)=\deg(v)=d=1$, then $I=(L)=(x_{i_1},x_{i_1+1},\dots,x_{{j_1}-1},x_{j_1})=\p_{[i_1,j_1]}$ is the standard primary decomposition of $I$. Thus, we may assume $d\ge2$.
	\item[-] If $I$ is initial, \emph{i.e.}, $I=J=(\mathcal{L}_t^i(v))$, we may assume $\min(v)=j_1\ge 2$, otherwise
	$$
	J=(x_1)\cap(\mathcal{L}_t(x_{1+t}\cdots x_{1+(d-1)t},v/x_{1})),
	$$
and to determine the standard primary decomposition of $J$, it suffices to determine that of $(\mathcal{L}_t(x_{1+t}\cdots x_{1+(d-1)t},v/x_{j_1}))$ which is an initial $t$--spread lexsegment ideal in the polynomial ring $K[x_{1+t},\dots,x_n]$ with fewer indeterminates than $S$.
	\item[-] If $I$ is final, \emph{i.e.}, $I=T=(\mathcal{L}_t^f(u))$, we can assume $\min(u)=i_1=1$. Indeed, if $i_1>1$, none of the variables $x_1,\dots,x_{i_1-1}$ divides any minimal monomial generator of $T$. So computing the standard primary decomposition of $T$ is equivalent to computing that of $T\cap K[x_{i_1},\dots,x_n]$ in the polynomial ring $K[x_{i_1},\dots,x_n]$ with fewer indeterminates than $S$.
\end{enumerate}
Hence, from now on, we assume that $I=(\mathcal{L}_t(u,v))=J\cap T$, with $u>_{\slex}v$, $\min(u)=1$, $\min(v) \ge 2$ and $\deg(u)=\deg(v)=d\ge 2$.\\

To compute the standard primary decomposition of $I=J\cap T$, we can proceed as follows. 
Firstly, we determine the decomposition of $J$, and if $I=J$ we are done. Secondly, we determine that of $T$, and if $I=T$ we are done. Finally, knowing the primary decompositions of $J$ and $T$ we take into account the intersection $J\cap T=I$, and deleting the non minimal primes, we obtain the standard primary decomposition of $I$ in the general case.

\subsection{The initial $t$--spread lexsegment case}
In this subsection we analyze the case of the standard primary decomposition of an initial $t$--spread lexsegment ideal of $S=K[x_1,\dots,x_n]$.

\begin{Thm}\label{primdecompinitialtspreadlex}
Let $v=x_{j_1}x_{j_2}\dots x_{j_d}$ be a $t$--spread monomial of degree $d\ge2$ of $S$ with $j_1\ge 2$ 
and let $J=(\mathcal{L}_t^i(v))$ be an initial $t$--spread lexsegment ideal of $S$. Then, the standard primary decomposition of $J$ is
$$
J\ =\ \bigcap_{p=1}^d\p_{F_p}\cap\bigcap_{F\in\mathcal{F}}\p_{[n]\setminus F},
$$
where
\begin{align*}
F_p\ &=[j_p]\setminus\supp_t(x_{j_1}x_{j_2}\cdots x_{j_{p-1}}),\ p=1,\dots,d,\\
\mathcal{F}\ &=\ \big\{\supp_t(w):w\in M_{n+1-t,d-1,t},\ w >_{\slex} v/x_{\max(v)}\big\}.
\end{align*}
\end{Thm}
\begin{proof}
	Let $\Delta$ be the simplicial complex  on the vertex set $[n]$ associated to $J$. 
	From Proposition \ref{prop:AlexDuality1}, we have to prove that the facets of $\Delta$ are exactly the sets $[n]\setminus F_{p}$, $p=1,\dots,d$, together with the sets of $\mathcal{F}$, \emph{i.e.},
	\begin{equation}\label{facetprimdecomplexinit}
	\mathcal{F}(\Delta)=\big\{[n]\setminus F_{p}:p=1,\dots,d\big\}\cup\mathcal{F}.
	\end{equation}
	Observe that if $G\in\Delta$, $G\ne\emptyset$, then $G$ is a facet of $\Delta$ if and only if $G\cup\{i\}\notin\Delta$, for all $i\in[n]\setminus G$, \emph{i.e.}, if and only if ${\bf x}_{G\cup\{i\}}\in J$, for all $i\in[n]\setminus G$.\smallskip
	
	We show that each $F\in\mathcal{F}$ is a facet of $\Delta$. Let $F\in\mathcal{F}$, then
	\begin{align*}
	F\ &=\ \supp_t(x_{\ell_1}\cdots x_{\ell_{d-1}})\ =\ \bigcup_{r=1}^{d-1}[\ell_r,\ell_r+(t-1)]\\
	&=\ \big\{\ell_1,\ell_1+1,\dots,\ell_1+(t-1),\ \dots,\ \ell_{d-1},\ell_{d-1}+1,\dots,\ell_{d-1}+(t-1) \big\},
	\end{align*}
	with $x_{\ell_1}\cdots x_{\ell_{d-1}}\in M_{n+1-t,d-1,t}$ and $x_{\ell_1}\cdots x_{\ell_{d-1}}>_{\slex}v/x_{\max(v)}=x_{j_1}\cdots x_{j_{d-1}}$.\medskip
	
	Clearly, ${\bf x}_F\notin J$, since  
	${\bf x}_F$ is not a multiple of any $t$--spread monomial of degree $d$ of $S$. To prove that $F$ is a facet, it suffices to show that ${\bf x}_{F\cup\{i\}}\in J$, for all $i\in[n]\setminus F$.\medskip
	
	Let $i\in[n]\setminus F$. We have: 
	\begin{enumerate}
		\item[-] if $i<\ell_1$, then ${\bf x}_{F\cup\{i\}}$ is divided by the $t$--spread monomial $x_ix_{\ell_1+(t-1)}\cdots x_{\ell_{d-1}+(t-1)}$;
		\item[-] if $\ell_j+(t-1)<i<\ell_{j+1}$, for some $1\le j\le d-2$, then ${\bf x}_{F\cup\{i\}}=$ ${\bf x}_Fx_i$ is divided by the $t$--spread monomial $\big(\prod\limits_{r=1}^jx_{\ell_r}\big)x_i\big(\prod\limits_{r=j+1}^{d-1}x_{\ell_r+(t-1)}\big)$;
		\item[-] if $\ell_{d-1}+(t-1)<i$, then ${\bf x}_{F\cup\{i\}}$ is is divided by the monomial $x_{\ell_1}\cdots x_{\ell_{d-1}}x_i$.
	\end{enumerate}
	
	More in detail, in every case, ${\bf x}_{F\cup\{i\}}$ is a multiple of a $t$--spread monomial of degree $d$ strictly greater than $v$, with respect to $>_{\slex}$. Therefore ${\bf x}_{F\cup\{i\}}\in J$ and $F$ is a facet.\medskip
	
	Now, we determine all the facets of $\Delta$. Let $G\in\Delta$ and define the following integers:
	\begin{enumerate}
	\item[] $\ell_1=\min(G)$, and
	\item[] $\ell_i=\min\big\{r\in G:r\ge\ell_{i-1}+t\big\}$, for $i\ge 2$.
	\end{enumerate}
	Let $\ell_k$ be the last element of the sequence $\ell_1 < \ell_2 < \cdots$.\\
	If $\ell_i\le j_i$ for all $i$, then $k\le d-1$, otherwise $x_{\ell_1}x_{\ell_2}\cdots x_{\ell_d}\in J$. Against the fact that $G\in\Delta$. Then,
	\begin{align*}
	G\ &\subseteq\ \bigcup_{i=1}^{k}[\ell_i,\ell_i+(t-1)]\cup\bigcup_{i=k+1}^{d-1}[j_i,j_i+(t-1)]\\
	&=\ \supp_t(x_{\ell_1}\cdots x_{\ell_k}x_{j_{k+1}}\cdots x_{j_{d-1}})=F \in\mathcal{F}.
	\end{align*}
	
	Assume there exists an integer $1\le p\le d$ such that $\ell_p>j_p$ and let $p$ be minimal for such a property. By the meaning of $p$, for $q<p$, $\ell_q\le j_q$. We need to distinguish two cases.\\
	\textsc{Case 1.} There exists an integer $q<p$ such that $\ell_q<j_q$.\\
	In such a case, $x_{\ell_1}x_{\ell_2}\cdots x_{\ell_{p-1}}>_{\slex}x_{j_1}x_{j_2}\cdots x_{j_{p-1}}$. Moreover, $k\le d-1$, otherwise $x_{\ell_1}x_{\ell_2}\cdots x_{\ell_{p-1}}x_{\ell_p}\cdots x_{\ell_d}>_{\slex}v$ and so ${\bf x}_G\in J$. A contradiction.
	Hence, 
	$$
	G\subseteq\bigcup_{i=1}^{k}[\ell_i,\ell_i+(t-1)]
	$$
	and we can quickly find a facet $F\in\mathcal{F}$ which contains $\bigcup_{i=1}^{k}[\ell_i,\ell_i+(t-1)]$ and consequently $G$.\\
	\textsc{Case 2.} For all integers $q<p$, $\ell_q=j_q$.\\
		In such a case, since $\ell_p>j_p$, we have $x_{j_1}\cdots x_{j_p}>_{\slex}x_{\ell_1}\cdots x_{\ell_p}$, and
	$$
	G\subseteq\bigcup_{i=1}^{p-1}[j_i,j_i+(t-1)]\cup[j_p+1,n]=[n]\setminus F_p
	$$
	and $[n]\setminus F_p$ is clearly  a facet. 
	Finally, all the facets of $\Delta$ are those described in (\ref{facetprimdecomplexinit}).
	
\end{proof}

\begin{Rem}\label{Rem:FacetsIniTSpread}
		\rm We note that for an initial $t$--spread lexsegment ideal $J=I_\Delta$, all the facets $F\in\mathcal{F}(\Delta)$ have cardinality $|F|\ge(d-1)t$. Indeed, from Theorem \ref{primdecompinitialtspreadlex}, $\mathcal{F}(\Delta)=\big\{[n]\setminus F_{p}:p=1,\dots,d\big\}\cup\mathcal{F}$. If $F\in\mathcal{F}$, then $|F|=(d-1)t$. Otherwise, if $F=[n]\setminus F_p$ for some $p$, then $|F|=|[n]\setminus F_p|=n-|F_p|=n-(j_p-(p-1)t)$. We observe that $j_p\le n-(d-p)t$. Thus
		$$
		|F|=n+(p-1)t-j_p\ge n+(p-1)t-n+(d-p)t=(d-1)t,
		$$
		as desired.
\end{Rem}

We illustrate the previous result with an example.
\begin{Expl}\label{ex:primdecompJ}
	\rm Let $v=x_2x_5x_7\in S=\mathbb{Q}[x_1,\dots,x_7]$ be a $2$--spread monomial and consider the initial $2$--spread ideal $J=(\mathcal{L}_2^i(v))$ of $S$:
	$$
	J=(x_{1}x_{3}x_{5},x_{1}x_{3}x_{6},x_{1}x_{3}x_{7},x_{1}x_{4}x_{6},x_{1}x_{4}x_{7},x_{1}x_{5}x_{7},
	x_{2}x_{4}x_{6},x_{2}x_{4}x_{7},x_{2}x_{5}x_{7}).
	$$
	We have $n=7,d=3,t=2$. The monomials $w\in M_{n+1-t,d-1,t}=M_{6,2,2}$ such that $w>_{\slex}v/x_{\max(v)}=x_2x_5$ are the following
	$$
	x_1x_3,\ \ x_1x_4,\ \ x_1x_5,\ \ x_1x_6,\ \ x_2x_4.
	$$
	Therefore,
	\begin{align*}
	\mathcal{F}\ &=\ \big\{\supp_2(x_1x_3),\supp_2(x_1x_4),\supp_2(x_1x_5),\supp_2(x_1x_6),\supp_2(x_2x_4)\big\}\\
	&=\ \big\{\{1,2,3,4\},\{1,2,4,5\},\{1,2,5,6\},\{1,2,6,7\},\{2,3,4,5\}\big\}.
	\end{align*}
	Moreover, $v=x_{j_1}x_{j_2}x_{j_3}=x_2x_5x_7$, hence
	\begin{align*}
	F_1&=[j_1]\setminus\supp_2(1)=[2]\setminus\emptyset=\{1,2\},\\
	F_2&=[j_2]\setminus\supp_2(x_{j_1})=[5]\setminus\{2,3\}=\{1,4,5\},\\
	F_3&=[j_3]\setminus\supp_2(x_{j_1}x_{j_2})=[7]\setminus\{2,3,5,6\}=\{1,4,7\}.
	\end{align*}
	Finally, the standard decomposition of $J$ is
	\begin{align*}
	J&=(x_{1},x_{2})\cap (x_{1},x_{4},x_{5})\cap(x_{1},x_{4},x_{7})\cap(x_{1},x_{6},x_{7})\cap(x_{3},x_{4},x_{5})\\
	&\phantom{=..}\cap
	(x_{3},x_{4},x_{7})\cap(x_{3},x_{6},x_{7})\cap(x_{5},x_{6},x_{7}).
	\end{align*}
\end{Expl}

Here are some corollaries of Theorem \ref{primdecompinitialtspreadlex}.
\begin{Cor}\label{cor:JinitialInvariants}
	In the hypotheses of Theorem \ref{primdecompinitialtspreadlex},  $\pd(S/J)=n-(d-1)t$, $\depth(S/J)=(d-1)t$ and $\dim(S/J)=n-j_1$.
\end{Cor}
\begin{proof}
	Since $J$ is a $t$--spread strongly stable ideal, by (\ref{eq1}), $\pd(S/J)=\pd(J)+1=n-(d-1)t$. So, by the Auslander--Buchsbaum formula, $\depth(S/J)=\depth(S)-\pd(S/J)=n-(n-(d-1)t)=(d-1)t$.\smallskip
	
	Let $\Delta$ be the simplicial complex associated to $J$. By Theorem \ref{primdecompinitialtspreadlex}, the facets of $\Delta$ are those described in (\ref{facetprimdecomplexinit}), and so, since $j_1\le n-(d-1)t$, and $j_1+(i-1)t\le j_i$, for $i=2,\dots,d$, we have
	\begin{align*}
	\alt(J)\ &=\ \min\big\{ n-(d-1)t, |F_\ell|\ :\ \ell=1,\dots,d\big\}\\
	&=\ \min\big\{ n-(d-1)t,j_1,j_2-t,\dots,j_{d}-(d-1)t\big\}\ =\ j_1.
	\end{align*}
	Therefore, $\dim(K[\Delta])=\dim(S/J)=n-j_1$.
\end{proof}
\begin{Cor}\label{cor:JCohenMac}
	An initial $t$--spread lexsegment ideal generated in degree $d\ge2$ is Cohen--Macaulay if and only if it is the $t$--spread Veronese ideal of degree $d$.
\end{Cor}
\begin{proof}
	Let $J=(\mathcal{L}_t^i(v))$, $v=x_{j_1}x_{j_2}\cdots x_{j_d}$ and let $\Delta$ be the simplicial complex associated to $J$. Assume $J$ is Cohen--Macaulay, then $\Delta$ is pure. By Theorem \ref{primdecompinitialtspreadlex}, the facets of $\Delta$ are those described in (\ref{facetprimdecomplexinit}), and so $\Delta$ is pure if and only if
	$$
	(d-1)t=n-j_1=n-(j_2-t)=\cdots=n-(j_d-(d-1)t).
	$$
	Hence $v=x_{n-(d-1)t}x_{n-(d-2)t}\cdots x_{n-t}x_n$ and $J=I_{n,d,t}$. Conversely, $I_{n,d,t}$ is Cohen--Macaulay. Indeed, $\dim(S/J)=\textup{depth}(S/J)=(d-1)t$ (Corollary \ref{cor:JinitialInvariants}).
\end{proof}
\subsection{The final $t$--spread lexsegment case}

In this subsection we determine the standard primary decomposition of a final $t$--spread lexsegment ideal of $S=K[x_1,\dots,x_n]$.

The next result will be crucial.

\begin{Prop}\label{propcardfacetslexfin}
Let $u$ be a $t$--spread monomial of degree $d$ of $S$ such that $\min(u)=1$ and let $T=(\mathcal{L}_t^f(u))$ be a final $t$--spread lexsegment of $S$. 
Let $\Delta$ be the simplicial complex on $[n]$ associated to $T$. 
Then $|F|\in\{(d-1)t,1+(d-1)t\}$, for all $F\in\mathcal{F}(\Delta)$.
\end{Prop}
\begin{proof}
	Since we may see $T$ as an initial $t$--spread lexsegment ideal with the order on the variables reversed $x_n>x_{n-1}>\cdots>x_1$, then each facet of $\Delta$ has cardinality $\ge(d-1)t$ (Remark \ref{Rem:FacetsIniTSpread}). To prove that $|F|\le 1+(d-1)t$ for all $F\in\mathcal{F}(\Delta)$, it is enough to show that each subset $G\subseteq[n]$ with cardinality $|G|=2+(d-1)t$ does not belong to $\Delta$, \emph{i.e.}, ${\bf x}_G\in T$. 
	
	Let $G\subseteq[n]$ with $|G|=2+(d-1)t$. Define
	
	\begin{enumerate}
	\item[] $\ell_1=\min\{g\in G:g>1\}$, and 
	\item[] $\ell_j=\min\big\{k\in G:k\ge \ell_{j-1}+t\big\}$, for $j\ge 2$.
	\end{enumerate}
	Since $|G|=2+(d-1)t$, the set $\{g\in G:g>1\}$ has at least $1+(d-1)t$ elements. Hence, the sequence of integers $\ell_1<\ell_2<\cdots$ has at least $d$ terms. Therefore, $H=\{\ell_1,\ell_2,\dots,\ell_d\}\subseteq G$ and $u>_{\slex}{\bf x}_H$. Indeed, $\min(u)=1$ and $\min({\bf x}_H)=\ell_1>1$ and so
	${\bf x}_H\in T$. It follows that ${\bf x}_G\in T$, as desired. 
\end{proof}

In the next theorem, we assume that $T\ne I_{n,d,t}$, as the case of the Veronese ideal has been covered in Theorem \ref{primdecompinitialtspreadlex}. 
\begin{Thm}\label{primdecompfinaltspreadlex} 
Let $u=x_{i_1}x_{i_2}\dots x_{i_d}$ be a $t$--spread monomial of degree $d\ge2$ of $S$ with $i_1=1$  and let $T=(\mathcal{L}_t^f(u))$ be a final $t$--spread lexsegment ideal of $S$, with $T\ne I_{n,d,t}$. Then, the standard primary decomposition of $T$ is
	$$
	T\ =\ \bigcap_{G\in\mathcal{G}}\p_{[n]\setminus G}\cap\bigcap_{H\in\mathcal{H}}\p_{[n]\setminus H},
	$$
	where
	\begin{align*}
	\mathcal{G}\ &=\ \big\{\cosupp_t(w)\cup\{1\}:x_1w\in M_{n,d,t},\ w>_{\slex} u/x_1\big\},\\
	\mathcal{H}\ &=\ \big\{\cosupp_t(w):x_1w\in M_{n,d,t},\ w\le_{\slex}u/x_1\big\}.
	\end{align*}
\end{Thm}
\begin{proof}
	Let $\Delta$ be a simplicial complex on the vertex set $[n]$ associated to $T$. By Proposition \ref{prop:AlexDuality1}, we have to prove that
	\begin{equation}\label{facetprimdecomplexfint}
	\mathcal{F}(\Delta)=\mathcal{G}\cup\mathcal{H}.
	\end{equation}
	
	By Proposition \ref{propcardfacetslexfin}, 
	the facets of $\Delta$ have cardinality $(d-1)t$ or $1+(d-1)t$.\smallskip
	
	Let $G\in\mathcal{F}(\Delta)$ such that $|G|=1+(d-1)t$. We prove that $G\in\mathcal{G}$. First of all, $\min(G)=1$. Indeed, if $\min(G)\ge 2$, then setting $s_1=\min(G)$, and 
	$$
	s_j=\min\big\{s\in G:s\ge s_{j-1}+t\big\},\,\, \mbox{for $j\ge2$},
	$$
	the sequence $s_1<s_2<\cdots$ has at least $d$ elements, otherwise $|G|<1+(d-1)t$. Hence, $\{s_1,s_2,\dots,s_d\}=U\subseteq G$ and ${\bf x}_U\in T$, as $\min({\bf x}_U)=\min(G)\ge 2>1=\min(u)$. It follows that $G\notin\Delta$. A contradiction. Therefore, $\min(G)=1$. 
	
	Consider the following integers:
	\begin{enumerate}
	\item[] $\ell_d =\max(G)$, and
	\item[] $\ell_j =\max\big\{\ell\in G:\ell\le\ell_{j+1}-t\big\}$, for $j<d$.
	\end{enumerate}
	The sequence $\ell_d>\ell_{d-1}>\dots>\ell_k$ has at least $d$ terms, otherwise $|G|<1+(d-1)t$. Moreover it has at most $d$ terms, otherwise $\ell_1>1$, $x_{\ell_1}x_{\ell_2}\cdots x_{\ell_d}\in T$ and then $G\notin\Delta$. A contradiction. Hence, $k=1$.
	
	Finally,
	\begin{align*}
	G\subseteq F\ &=\ \big\{\ell_d,\ell_d-1,\dots,\ell_d-(t-1),\ \ldots,\ \ell_2,\ell_2-1,\dots,\ell_2-(t-1),\ \ell_1 \big\}\\
	&=\ \ \bigcup_{r=2}^{d}[\ell_r-(t-1),\ell_r]\cup\{\ell_1\}\ =\ \cosupp_t(x_{\ell_2}\cdots x_{\ell_d})\cup\{\ell_1\}.
	\end{align*}
	Moreover $1+(d-1)t=|G|\le|F|=1+(d-1)t$ and $\ell_1=1$, and so $G=F\in\mathcal{G}$. Clearly, $G$ is a facet. Indeed ${\bf x}_G\notin T$.\\
	
	Now, let us determine the facets of $\Delta$ with cardinality $(d-1)t$. Let $H\in\mathcal{F}(\Delta)$ with $|H|=(d-1)t$. We prove that $H\in\mathcal{H}$. Consider the following integers
	\begin{enumerate}
	\item[] $\ell_d =\max(H)$, and
	\item[] $\ell_j=\max\big\{\ell\in H:\ell\le\ell_{j+1}-t\big\}$,  for $j<d$.
	\end{enumerate}
	The sequence $\ell_d>\ell_{d-1}>\dots>\ell_k$ has at least $d-1$ terms, otherwise $|H|<(d-1)t$, against our assumption. In fact, we have $k=2$. Indeed, suppose $k=1$. 
		If $\ell_1=1$, then $H\subseteq G\in\mathcal{G}$. A contradiction since $H$ is a facet. On the other hand, if $\ell_1>1$, then $x_{\ell_1}x_{\ell_2}\cdots x_{\ell_d}\in T$, and $H\notin\Delta$. A contradiction.
	
	Finally, setting $U=\{\ell_2,\dots,\ell_d\}\subseteq H$, if $x_1{\bf x}_U>_{\slex}u$, then $H$ is contained in a facet $G\in\mathcal{G}$. A contradiction since $H$ is a facet. If $x_1{\bf x}_U\le_{\slex}u$, then $H\in\mathcal{H}$ and $H$ is clearly a facet. The proof is complete.
\end{proof}

\begin{Expl}\label{ex:primdecompT}
	\rm Let $u= x_1x_4x_6$ be a $2$--spread monomial of degree $3$ of $S=\mathbb{Q}[x_1,\dots,x_7]$. Consider the ideal $T=(\mathcal{L}_2^f(u))$ of $S$:
		$$
	T=(x_{1}x_{4}x_{6},x_{1}x_{4}x_{7},x_{1}x_{5}x_{7},x_{2}x_{4}x_{6},x_{2}x_{4}x_{7},x_{2}x_{5}x_{7},
	x_{3}x_{5}x_{7}).
	$$ 
	The monomials $w\in  M_{7,2,2}$ such that $x_1w\in M_{7,3,2}$ and  $w>_{\slex} u/x_1 = x_4x_6$ are
		$$
	x_3x_5,\ \ x_3x_6,\ \ x_3x_7;
	$$
	whereas, the monomials $w\in  M_{7,2,2}$  such that $x_1w\in M_{7,3,2}$ and $w\le_{\slex}u/x_1 = x_4x_6$ are 
	$$
	x_4x_6,\ \ x_4x_7,\ \ x_5x_7.
	$$
	Therefore,
	\begin{align*}
	\mathcal{G}\ &=\ \big\{\cosupp_2(x_3x_5)\cup\{1\}, \cosupp_2(x_3x_6)\cup\{1\}, \cosupp_2(x_3x_7)\cup\{1\}\big\}\\
	&=\ \big\{\{1,2,3,4,5\},\{1,2,3,5,6\},\{1,2,3,6,7\} \big\},\medskip\\
	\mathcal{H}\ &=\ \big\{\cosupp_2(x_4x_6), \cosupp_2(x_4x_7), \cosupp_2(x_5x_7)\big\}\\
	&=\ \big\{\{3,4,5,6\},\{3,4,6,7\},\{4,5,6,7\} \big\}.
	\end{align*}
	Thus, the standard primary decomposition of $T$ is
	$$
	T=(x_{4},x_{5})\cap(x_{4},x_{7})\cap(x_{6},x_{7})\cap(x_{1},x_{2},x_{3})\cap(x_{1},x_{2},x_{5})\cap(x_{1},x_{2},x_{7}).
	$$
\end{Expl}

Here are some corollaries of Theorem \ref{primdecompfinaltspreadlex}.
\begin{Cor}\label{cor:TinitialInvariants}
	Let $T=(\mathcal{L}_t^f(u))$ be a final $t$--spread lexsegment ideal of $S$. Then $\pd(S/T)=n-(d-1)t$, $\depth(S/T)=(d-1)t$ and
	$$
	\dim(S/T)=\begin{cases}
	(d-1)t&\textup{if}\ u=x_1x_{1+t}\cdots x_{1+(d-1)t},\\
	1+(d-1)t&\textup{otherwise}.
	\end{cases}
	$$
\end{Cor}
\begin{proof}
	The ideal $T$ is $t$--spread strongly stable but with the order on the variables reversed, $x_n>x_{n-1}>\dots>x_1$. Hence, by (\ref{Bettinumbersreversetspread}), as $\min(u)=1$ and $u\in T$, $\pd(S/T)=\pd(T)+1=n-1-(d-1)t+1=n-(d-1)t$. By the Auslander--Buchsbaum formula, $\depth(S/T)=(d-1)t$.
	
	Suppose $T\ne I_{n,d,t}$. Let $\Delta$ the simplicial complex associated to $T$. By Theorem \ref{primdecompfinaltspreadlex}, the facets of $\Delta$ are those in (\ref{facetprimdecomplexfint}). Now, $\mathcal{H}$ is always non empty, as $u\in\mathcal{L}_t^f(u)$. Moreover, $\mathcal{G}$ is non empty too, as $u\ne\min(M_{n,d,t})$, otherwise $T=I_{n,d,t}$. Thus
	\begin{align*}
	\alt(T)\ &=\ \min\big\{\alt(\p_{[n]\setminus G}),\alt(\p_{[n]\setminus H}):G\in\mathcal{G},H\in\mathcal{H}\big\}\\
	&=\ \min\big\{n-(d-1)t,n-1-(d-1)t\big\}\ =\ n-1-(d-1)t.
	\end{align*}
	Thus, $\dim(K[\Delta])=1+(d-1)t$.\\
	Otherwise, if $T=I_{n,d,t}$, \emph{i.e.}, $u=\max(M_{n,d,t})=x_1x_{1+t}\cdots x_{1+(d-1)t}$, we have $\dim(K[\Delta])$ $=(d-1)t$ (Corollary \ref{cor:JinitialInvariants}).
\end{proof}
\begin{Cor}\label{cor:TCohenMac}
	A final $t$--spread lexsegment ideal of $S$ generated in degree $d\ge2$ is Cohen--Macaulay if and only if it is the $t$--spread Veronese ideal of degree $d$.
\end{Cor}
\begin{proof}
	Let $T=(\mathcal{L}_t^f(u))$ and let $\Delta$ be the simplicial complex associated to $T$. Assume $T$ is Cohen--Macaulay, then $\Delta$ is pure. By contradiction, suppose $T\ne I_{n,d,t}$. By Theorem \ref{primdecompfinaltspreadlex}, the facets of $\Delta$ are those described in (\ref{facetprimdecomplexfint}). The families $\mathcal{G}$ and $\mathcal{H}$ are both non empty, the first as $u\in\mathcal{L}_t^f(u)$, the second as $u\ne\max(M_{n,d,t})$. But, any $G\in\mathcal{G}$ has cardinality $|G|=1+(d-1)t$ and any $H\in\mathcal{H}$ has cardinality $|H|=(d-1)t$. So $\Delta$ should not be pure. A contradiction. Thus $T=I_{n,d,t}$. The converse easily follows.
\end{proof}
\begin{Rem}\label{rem:unmixed}
 \em A simplicial complex $\Delta$ is pure if and only if the Stanley--Reisner ideal $I_\Delta$ is unmixed, in the sense that all associated prime ideals have the same height (see, for instance, \cite{JT}).
One can observe that the proofs of Corollary \ref{cor:JCohenMac} and Corollary \ref{cor:TCohenMac} have pointed out that the classification for initial Cohen--Macaulay $t$--spread lexsegment ideal  and for final Cohen--Macaulay $t$--spread lexsegment ideal  is equivalent to the classification for initial unmixed $t$--spread lexsegment ideal and for final unmixed $t$--spread lexsegment ideal.
\end{Rem}
\subsection{The case of a completely $t$--spread lexsegment ideal}

The aim of this subsection is to determine the minimal primary decomposition of a completely $t$--spread lexsegment ideal $I=(\mathcal{L}_t(u,v))$ of $S=K[x_1,\dots,x_n]$, $u, v \in M_{n,d,t}$, $u>_{\slex} v$ with $\min(u)=1$ and $\min(v)\ge 2$.

\begin{Thm}\label{primdecompgeneraltspreadlex}
	Let $u=x_{i_1}x_{i_2}\dots x_{i_d}, v=x_{j_1}x_{j_2}\cdots x_{j_d}$ be $t$--spread monomials of degree $d\ge2$ of $S$ with $i_1=1$ and $j_1\ge 2$. Let $I=(\mathcal{L}_t(u,v))$ be a completely $t$--spread lexsegment ideal of $S$, $I\ne I_{n,d,t}$. Then, the standard primary decomposition of $I$ is
	$$
	I\ =\ \bigcap_{G\in\mathcal{G}}\p_{[n]\setminus G}\cap\bigcap_{p\in\mathcal{I}}\p_{F_p}\cap\bigcap_{F\in\widetilde{\mathcal{F}}}\p_{[n]\setminus F},
	$$
	where $F_p$ $(p=1,\dots,d)$ and $\mathcal{G}$ 
	are the families described in Theorems \ref{primdecompinitialtspreadlex} and \ref{primdecompfinaltspreadlex},
	\begin{enumerate}
		\item[-] $\mathcal{I}=[d]\setminus\big\{p\in[d]:\big|[n]\setminus F_{p}\big|=(d-1)t \,\,\mbox{and  $v/x_{j_p} >_{\slex} u/x_1$}\big\}$,
		\item[-] $\widetilde{\mathcal{F}}$ is the family consisting of the sets $F=\supp_t(w)\in\mathcal{F}$ with $w=x_{\ell_1}\cdots x_{\ell_{d-1}}$ such that 
		\begin{enumerate}
		\item[-] $1\notin F$ and $x_{\ell_1+(t-1)}\cdots x_{\ell_{d-1}+(t-1)}\le_{\slex}u/x_1$, or
		\item[-] $1\in F$ and $F\cup\{j\}\notin\mathcal{G}$, for all $j\notin F$,
	\end{enumerate}
	\end{enumerate}
	where $\mathcal{F}$ is the family described in Theorem \ref{primdecompinitialtspreadlex}.
	
\end{Thm}
\begin{proof}
	By hypothesis, $I=J\cap T$ with $J=(\mathcal{L}^i_t(v))$ and $T=(\mathcal{L}_t^f(u))$. By Theorems \ref{primdecompinitialtspreadlex} and \ref{primdecompfinaltspreadlex} we have
	\begin{equation}\label{eq:notminimalprimdecom}
	I=T\cap J=\bigcap_{G\in\mathcal{G}}\p_{[n]\setminus G}\cap\bigcap_{H\in\mathcal{H}}\p_{[n]\setminus H}\cap\bigcap_{p=1}^d\p_{F_p}\cap\bigcap_{F\in\mathcal{F}}\p_{[n]\setminus F}.
	\end{equation}
	To find the standard primary decomposition of $I$, it suffices to delete from this presentation those primes that are not minimal. Equivalently, if $\Delta$ is the simplicial complex on vertex set $[n]$ associated to $I$, we must determine the facets of $\Delta$. By (\ref{eq:notminimalprimdecom}) we have that
	\begin{equation}\label{eq:4families}	
	\mathcal{F}(\Delta)\subseteq\mathcal{G}\cup\mathcal{H}\cup\big\{[n]\setminus F_p:p=1,\dots,d\big\}\cup\mathcal{F}.
	\end{equation}
	where $\mathcal{G}$, $\mathcal{H}$, $F_p$ ($p=1,\dots,d$), $\mathcal{F}$
 are the families described in Theorems \ref{primdecompinitialtspreadlex} and \ref{primdecompfinaltspreadlex}.\\\\
	\textsc{Claim.} $\mathcal{F}(\Delta)=\mathcal{G}\cup\big\{[n]\setminus F_{\ell}:\ell\in\mathcal{I}\big\}\cup\widetilde{\mathcal{F}}$.\\
	\textsc{Proof of the Claim}.
	We analyze the four families in (\ref{eq:4families}), and in so doing we determine the facets of $\Delta$.\\
	\begin{enumerate}
	\item[-] Let $G\in\mathcal{G}$. Since $I\subseteq T$ and ${\bf x}_G\notin T$, it follows that ${\bf x}_G\notin I$, \emph{i.e.}, $G\in\Delta$. Let us prove that $G$ is a facet. Indeed, $G\not\subseteq F, H$, for all $F\in\mathcal{F}$ and all $H\in\mathcal{H}$. Indeed, $|F|=|H|=(d-1)t<1+(d-1)t=|G|$. Moreover, $G\not\subseteq [n]\setminus F_{p}$, as $1\in G$ and $1\in F_{p}$. In fact, $\min(v)=j_1\ge2$.\\
	\item[-] Let $H\in\mathcal{H}$. Arguing as before, $H\in\Delta$. We show that we can ``eliminate" the family $\mathcal{H}$. In fact, we can write $H=\cosupp_t(x_{\ell_2}\cdots x_{\ell_d})$, with $x_1x_{\ell_2}\cdots x_{\ell_d}\in M_{n,d,t}$ and $x_{\ell_2}\cdots x_{\ell_d}\le_{\slex}u/x_1$. Let
	$$
	\widetilde{H}=\{\ell_2-(t-1),\ell_3-(t-1),\ldots,\ell_{d}-(t-1)\}\subseteq H.
	$$
	
Setting $h_i =  \ell_{i+1}-(t-1)$, for $i=1, \ldots, d-1$, then $\widetilde H=\{h_1<h_2<\dots<h_{d-1}\}$. We distinguish three cases. 
\begin{enumerate}
\item[-] \textsc{Case 1.} Let ${\bf x}_{\widetilde{H}}>_{\slex}v/x_{\max(v)}$. Since $\supp_t({\bf x}_{\widetilde{H}}) = \cosupp_t(x_{\ell_2}\cdots x_{\ell_d})=H$, then $H\in\mathcal{F}$ (we will determine when $H\in\mathcal{F}$ is a facet of $\Delta$ later).

\item[-] \textsc{Case 2.} Let ${\bf x}_{\widetilde{H}}<_{\slex}v/x_{\max(v)}$. Then there exists an integer $s\in[d-1]$ such that $j_1=h_1,\ldots,j_{s-1}=h_{s-1}$ and $j_s<h_s$. Hence
	\begin{align*}
	H\ &\subseteq\ \big\{j_1,\dots,j_1+(t-1),\ldots,j_{s-1},\dots,j_{s-1}+(t-1),\ j_s+1,\dots,n\big\}\\
	&=\ [n]\setminus F_s=F.
	\end{align*}
	Moreover, ${\bf x}_F\notin I$, as $I\subseteq J$ and ${\bf x}_F\notin J$. So $H$ is included in $F\in\Delta$.
		\item[-]  \textsc{Case 3.} Let ${\bf x}_{\widetilde{H}}=v/x_{\max(v)}$. Then we have $H\subseteq[n]\setminus F_d=F$, and also $F\in\Delta$. Thus $H$ is included in the face $F\in\mathcal{F}$.
	 \end{enumerate}
	\item[-] Let $F=[n]\setminus F_{p}$, some some $1\le p\le d$. Since $I\subseteq J$ and ${\bf x}_F\notin J$, then ${\bf x}_F\notin I$ and so $F\in\Delta$. If $|F|\ge1+(d-1)t$, then $F$ is a facet as $F$ is not contained in any other set of the families $\mathcal{F},\mathcal{H}$. Moreover, $F\ne G$ for all $G\in\mathcal{G}$, as $1\in G\setminus F$.\\ 
	Suppose $|F|=(d-1)t$. Then $F$ is not a facet if and only if $F\subsetneq G$, for some $G\in\mathcal{G}$. Therefore, if and only if $F=G\setminus\{1\}$, for some $G\in\mathcal{G}$. On the other hand, $G=\big\{\ell_d,\ell_d-1,\dots,\ell_d-(t-1),\ \ldots,\ \ell_2,\ell_2-1,\dots,\ell_2-(t-1),\ 1 \big\}$ is characterized by the conditions $x_1x_{\ell_2}\cdots x_{\ell_d}\in M_{n,d,t}$ and $x_1x_{\ell_2}\cdots x_{\ell_d}>_{\slex}u$. We observe that $\supp(v/x_{j_p})\subseteq F=[n]\setminus F_{p}$. Hence, if $F\subseteq G$, then
	$$
	v/x_{j_p}\ge_{\slex}x_{\ell_2}\cdots x_{\ell_d}>_{\slex}u/x_1.
	$$
	Thus, if $F\subseteq G$, then $v/x_{j_p}>_{\slex}u/x_1$. On the other hand, if $v/x_{j_p}\le_{\slex}u/x_1$, then $x_{\ell_2}\cdots x_{\ell_d}\le_{\slex}v/x_{j_p}\le_{\slex}u/x_1$ and $G$ is not a facet.
	Finally, we have verified that $F=[n]\setminus F_{p}$ is a facet if and only if $p\in\mathcal{I}$.\\
	\item[-] Let $F\in\mathcal{F}$. Clearly ${\bf x}_F\notin I$ and so $F\in\Delta$. We have that $$F=\big\{\ell_1,\ell_1+1,\ \ldots,\ \ell_1+(t-1),\dots,\ell_{d-1},\ell_{d-1}+1,\dots,\ell_{d-1}+(t-1)\big\},$$
	with $x_{\ell_1}\cdots x_{\ell_{d-1}}\in M_{n+1-t,d-1,t}$ and $x_{\ell_1}\cdots x_{\ell_{d-1}}>_{\slex}v/x_{\max(v)}$. If $F$ is not a facet, then $F\subsetneq G$, for some $G\in\mathcal{G}$.
	\begin{enumerate}
	\item[-] \textsc{Case 1.} Let $1\notin F$. Then we have $F=G\setminus\{1\}$. $G$ is characterized by the condition $x_{\ell_1+(t-1)}\cdots x_{\ell_{d-1}+(t-1)}>_{\slex}u/x_1$. Hence, in such a case, $F$ is a facet if and only if $x_{\ell_1+(t-1)}\cdots x_{\ell_{d-1}+(t-1)}\le_{\slex}u/x_1$.\\
	\item[-]  \textsc{Case 2.} Let $1\in F$. Set $\ell=\ell_j+t$, if $j=\min\big\{j\in [d-2]:\ell_j+(t-1)<\ell_{j+1}\big\}$ does exist, otherwise set $\ell=\ell_{d-1}+t$. Then $F\subseteq F\cup\{\ell\}$.\\ 
\end{enumerate}
	Finally $F$ is a facet if and only if $F\cup\{\ell\}\notin\mathcal{G}$.
		Thus, if and only if
	\begin{align*}
	x_{\ell_1+t}\cdots x_{\ell_{j-1}+t}x_{\ell_j+t}\cdots x_{\ell_{d-1}+(t-1)}\le_{\slex}u/x_1,&\ \ \text{if}\ \ell=\ell_j+t,\\
	x_{\ell_1+t}\cdots x_{\ell_{j-1}+t}x_{\ell_j+t}\cdots x_{\ell_{d-1}+t}\le_{\slex}u/x_1,&\ \ \text{if}\ \ell=\ell_{d-1}+t.
	\end{align*}
	\end{enumerate}
The claim follows.
\end{proof}

\begin{Expl}
	\rm Let $u=x_1x_4x_6$ and $v=x_2x_5x_7$ be $2$--spread monomials of degree $3$ of $S=\mathbb{Q}[x_1,\dots,x_7]$. Consider the $t$--spread lexsegment ideal 
	$$
	I=(\mathcal{L}_2(u,v))= (x_1x_4x_6,x_1x_4x_7,x_1x_5x_7,x_2x_4x_6,x_2x_4x_7,x_2x_5x_7)
	$$
	of $S$. Firstly, we verify that $I$ is a completely $2$--spread lexsegment ideal, \emph{i.e.}, $I=J\cap T$, where $J$ and $T$ are the $2$--spread lexsegment ideals in Examples \ref{ex:primdecompJ} and \ref{ex:primdecompT}, respectively. By \cite[Theorem 3.7. (b)]{FC1},  $I$ is a completely $2$--spread lexsegment ideal if and only if for all $w\in M_{7,3,2}$ with $w<_{\slex}v=x_2x_5x_7$, there exists $s\in\supp(w)$ such that $x_{\min(u)}w/x_s=x_1w/x_s\le_{\slex}u=x_1x_4x_6$.\\
	We have $w=x_3x_5x_7$ and we may choose $s=\min(w)$. Thus, $I$ is a completely $2$--spread lexsegment ideal.\\
	Setting, $u=x_{i_1}x_{i_2}x_{i_3} = x_1x_{i_2}x_{i_3}$ and $v=x_{j_1}x_{j_2}x_{j_3} = x_2x_{j_2}x_{j_3}$, we have
	$$
	\mathcal{I}=[3]\setminus\big\{p\in[3]:\big|[7]\setminus F_{p}\big|=2\cdot2\ \text{and}\ u/x_1=x_4x_6<_{\slex}v/x_{j_p}=x_2x_5x_7/x_{j_{\ell}}\big\}.
	$$
	
	The only sets $F_{p}$ with $\big|[7]\setminus F_{p}\big|=2\cdot2=4$ are $F_2=\{1,4,5\}$ and $F_3=\{1,4,7\}$ (Example \ref{ex:primdecompJ}). Moreover, $v/x_{j_2}=x_2x_7>_{\slex}x_4x_6$ and $v/x_{j_3}=x_2x_5>_{\slex}x_4x_6$. Thus, $\mathcal{I}=[3]\setminus\{2,3\}=\{1\}$.
	
	From Examples \ref{ex:primdecompJ} and \ref{ex:primdecompT}, we have:
	 $$\mathcal{G}=\big\{\{1,2,3,4,5\},\{1,2,3,5,6\},\{1,2,3,6,7\} \big\}$$ and
	$$
	\mathcal{F}=\big\{\{1,2,3,4\},\{1,2,4,5\},\{1,2,5,6\},\{1,2,6,7\},\{2,3,4,5\}\big\}.
	$$
	Let $F\in\mathcal{F}$ such that $1\notin F$. Then $F=\{2,3,4,5\}=\supp_2(x_2x_4)$. We have that 
	$x_3x_5>_{\slex}u/x_1=x_4x_6$, so $F\notin\widetilde{\mathcal{F}}$.\\
	Let $F\in\mathcal{F}$ such that $1\in F$. Then $F\in\widetilde{\mathcal{F}}$ if and only if $F\not\subseteq G$, for all $G\in\mathcal{G}$. Since such sets do not exist, $\widetilde{\mathcal{F}}=\emptyset$.

		Finally, the standard primary decomposition of $I$ is
	$$
	I=(x_1,x_2)\cap(x_4,x_5)\cap(x_4,x_7)\cap(x_6,x_7).
	$$
\end{Expl}

\section{Cohen--Macaulay $t$--spread lexsegment ideals} \label{sec3}

In Section \ref{sec2}, we have seen that if $I$ is an initial or a final $t$--spread lexsegment ideal of $S=K[x_1,\dots,x_n]$, $t\ge1$, then $I$ is Cohen--Macaulay if and only if $I=I_{n,d,t}$ is the Veronese $t$--spread ideal of degree $d$ of $S$ (Corollaries \ref{cor:JCohenMac}, \ref{cor:TCohenMac}). In this section, we consider the more general problem of classifying all Cohen--Macaulay $t$--spread lexsegment ideals $I=(\mathcal{L}_t(u,v))$ of $S$, with $t\ge1$. The case $t=1$ has been examined also in \cite{BST} by using Serre's condition $(S_2)$. 

As in Section \ref{sec2}, we may assume that $I=(\mathcal{L}_t(u,v))$ with $u>_{\slex}v$, $\min(u)=1$, $\min(v)\ge 2$ and $\deg(u)=\deg(v)=d\ge 2$. In order to achieve our purpose, we will distinguish two cases: $\min(v)=2$ and $\min(v)>2$. 

First, we note that we can always suppose $n>2+(d-1)t$. Indeed, for $n=1+(d-1)t$, there is only one $t$--spread lexsegment ideal, namely $I=(x_1x_{1+t}\cdots x_{1+(d-1)t})=I_{1+(d-1)t,d,t}$ which is always a Cohen--Macaulay ideal; whereas, for $n=2+(d-1)t$, $I=(\mathcal{L}_t(u,v))$ with $v=x_2x_{2+t}\cdots x_{2+(d-1)t}$, necessarily. In fact, in such a case 
there does exist only one $t$--spread monomial whose minimum is equal to $2$. Hence, $I$ is final and therefore Cohen--Macaulay if and only if $I=I_{2+(d-1)t,d,t}$ (Corollary \ref{cor:TCohenMac}).

In what follows we denote the squarefree lex order $\ge_{\slex}$ simply by $\ge$.

\subsection{The height two case}

In this subsection we study the Cohen--Macaulayness of a $t$--spread lexsegment ideal $I=(\mathcal{L}_t(u,v))$ of $S$ with $\min(v)=2$.  We will distinguish two cases: $3+(d-1)t\le n\le 3+(2d-3)t$ and $n\ge 4+(2d-3)t$.

We start with the case $3+(d-1)t\le n\le 3+(2d-3)t$. The notion of \textit{Betti splitting} \cite{EK} (see also \cite{FHT2009} and the reference therein) will be crucial.

Let $I$, $P$, $Q$ be monomial ideals of $S=K[x_1,\dots,x_n]$ such that $G(I)$ is the disjoint union of $G(P)$ and $G(Q)$. We say that $I=P+Q$ is a \textit{Betti splitting} if
$$
\beta_{i,j}(I)=\beta_{i,j}(P)+\beta_{i,j}(Q)+\beta_{i-1,j}(P\cap Q), \ \ \ \textup{for all}\ i,j.
$$ 
If $I=P+Q$ is a Betti splitting \cite[Corollary 2.2]{FHT2009}, then
\begin{align}
\label{eq:pd(I)j1=2BettiSplit}\pd(I)\ &=\ \max\{\pd(P),\pd(Q),\pd(P\cap Q)+1\}.
\end{align}

\begin{Lem}\label{Lem:BettiSplitxi}
	\textup{\cite[Corollary 2.7]{FHT2009}} Let $I$ be a monomial ideal of $S$. Let $Q$ be the ideal generated
by all elements of $G(I)$ divisible by $x_i$ and let $P$ be the ideal generated by all other elements of $G(I)$. If the ideal $Q$ has a linear resolution, then $I=P+Q$ is a Betti splitting.
\end{Lem}

For $I$ a monomial ideal of $S$, we set $\gcd(I)=\gcd(u:u\in G(I))$. 

\begin{Thm}\label{Thm:gcd(I)PcapQ}
	Let $u=x_{i_1}x_{i_2}\dots x_{i_d}$, $v=x_{j_1}x_{j_2}\cdots x_{j_d}$ be $t$--spread monomials of degree $d\ge 2$ of $S$ with $i_1=1$, $j_1=2$ and let $I$ be the $t$--spread lexsegment ideal $I=(\mathcal{L}_t(u,v))$ of $S$. Assume $3+(d-1)t\le n\le 3+(2d-3)t$ and set
	$$
	P=(\mathcal{L}_t(u,x_1x_{n-(d-2)t}\cdots x_{n-t}x_n)),\ \ \ \ Q=(\mathcal{L}_t(x_2x_{2+t}\cdots x_{2+(d-1)t},v)).
	$$
	Then $I$ is Cohen--Macaulay if and only if $\gcd(I)=1$ and $P\cap Q$ is a principal ideal.
\end{Thm}
\begin{proof}
	Let $I$ be Cohen--Macaulay and let $\Delta$ be the simplicial complex on the vertex set $[n]$ associated to $I$. Then $\Delta$ is pure. Since each monomial in $G(I)$ has minimum equal to 1 or 2, the set $[n]\setminus\{1,2\}$ is a facet of $\Delta$. Thus, $\Delta$ must be pure of dimension $\dim(\Delta)=n-3$.
	
	Now, we prove that $v\in \{v_1, v_2\}$, where
	\begin{equation}\label{moncase2.2}
	\begin{aligned}
	v_1&=\textstyle\phantom{\big(}\prod_{s=0}^{d-1}x_{2+st},\\
	v_2&=\textstyle\big(\prod_{s=0}^{d-2}x_{2+st}\big)x_{3+(d-1)t}.
	\end{aligned}
	\end{equation}
	First, note that $v_1>v_2$ are the greatest monomials of $M_{n,d,t}$ with minimum 2. So $v_1\in I$. Let $E=\big\{\max(w):w\in\mathcal{L}_t(v_1,v)\big\}$ and set $k=\max(E)$. Then $k\ge 2+(d-1)t$. Indeed $v_1\in I$. If $v=v_1$, there is nothing to prove. Assume $v>v_1$, then $|E|\ge2$. We show that $v=v_2$. The monomial prime ideal $\p_{\{1\}\cup E}$ contains $I$, and $\textup{height}(\p_{\{1\}\cup E})=1+|E|\ge 3$. Since $I$ is Cohen--Macaulay, thus height--unmixed of height two, there must be a minimal prime $\p\in\Min(I)$ of height $2$ that contains properly $\p_{\{1\}\cup E}$. For all $j\in E$, $\p_{\{1\}\cup(E\setminus\{j\})}$ does not contain $I$. Indeed, the monomial $\big(\prod_{s=0}^{d-2}x_{2+st}\big)x_j\in\mathcal{L}_t(v_1,v)$ is not in $\p_{\{1\}\cup(E\setminus\{j\})}$. Hence, $x_1$ can be omitted and $\p_{E}$ must contain $I$. One can quickly observe that no variable $x_j$, $j\in E$, can be omitted again. Hence, $\p_{E}$ is a minimal prime with $\textup{height}(\p_{E})=|E|=2$, \emph{i.e.}, $k=2+(d-1)t$. So $v\in\{v_1,v_2\}$ as desired.\\
	
	For $\ell=0,\dots,d-1$, let us define the following monomials:
	\begin{equation}\label{moncase2.3}
	\begin{aligned}
	u_\ell\ &=\ x_1\big(\textstyle\prod_{s=\ell}^{d-2}x_{n-st-1}\big)\big(\prod_{s=0}^{\ell-1}x_{n-st}\big)\\
	&=\ x_1x_{n-(d-2)t-1}\cdots x_{n-\ell t-1}x_{n-(\ell-1)t}\cdots x_{n-t}x_n.
	\end{aligned}
	\end{equation}
	
	One can observe that $\mathcal{L}_t(u_0,u_{d-1})=\{u_0>u_1>\dots>u_{d-1}\}$ and furthermore $u_{d-1}$ is the smallest monomial of $M_{n,d,t}$ whose minimum is equal to $1$. Thus $u>u_{d-1}>v$ and $u_{d-1}\in I$. 
We prove that $u=u_\ell$, for some $\ell\in\{0,\dots,d-1\}$.\\
Indeed, writing $u=x_1x_{i_2}\cdots x_{i_d}$, we note that $i_2\le n-(d-2)t$. The condition $u\le u_0$ is equivalent to $i_2\ge n-(d-2)t-1$. As before, we consider a suitable monomial prime ideal that contains $I$, namely $\p_{[i_2,n-(d-2)t]\cup\{2\}}$. None of $x_j$ with $j\in[i_2,n-(d-2)t]$ can be omitted from $\p_{[i_2,n-(d-2)t]\cup\{2\}}$. Thus, $\p_{[i_2,n-(d-2)t]}$ must be a minimal prime of height two. Hence, $i_2=n-(d-2)t-1$, as desired.

By what shown so far, if $I$ is Cohen--Macaulay, then $I=(\mathcal{L}_t(u_\ell,v_k))$, for some $\ell\in\{0,\dots,d-1\}$ and $k\in\{1,2\}$. 
On the other hand, $I$ is Cohen--Macaulay if and only if $\dim(S/I)=\textup{depth}(S/I)$. We note that $\dim(S/I)=n-2$. Indeed, we have shown previously that $\textup{height}(I)=2$.
By the Auslander--Buchsbaum formula, $\textup{depth}(S/I)=n-\pd(S/I)=n-1-\pd(I)$. Hence, $I$ is Cohen--Macaulay if and only if $\textup{height}(I)=2$ and $\pd(I)=1$. 

Hence, in order to establish the theorem we need to prove the next facts.\\
\textsc{Claim 1.} $\height(I)=2$ if and only if $\gcd(I)=\gcd(w:w\in G(I))=1$.\\ 
\textsc{Proof of Claim 1.}  Indeed, $\p_{[2]}=(x_1,x_2)$ is a minimal prime of $I$ having height two, so $\height(I)\le\height(\p_{[2]})=2$. Moreover, $\height(I)=1$ if and only if there is a minimal prime $\mathfrak{q}$ of $I$ whose height is $1$, \emph{i.e.}, there exists a variable $x_s$ dividing all generators of $I$. The claimed statement follows.\\
\textsc{Claim 2.} $\pd(I)=1$ if and only if the ideal $(\mathcal{L}_t(u,u_{d-1}))\cap(\mathcal{L}_t(v_1,v))=P\cap Q$ is a principal ideal.\smallskip\\
\textsc{Proof of Claim 2.} We can observe that
	$$
	P=(w\in G(I):\min(w)=1),\ \ \ Q=(w\in G(I):\min(w)=2).
	$$
	Clearly $I=P+Q$. We claim that $I=P+Q$ is a Betti splitting of $I$. Indeed using Lemma \ref{Lem:BettiSplitxi}, it sufficies to note that $Q=(w\in G(I):x_2\ \textup{divides}\ w)$ and that $Q$ has a linear resolution. For all monomials $w\in G(I)$ with $\min(w)=1$, $\min(w/x_1)\ge n-(d-2)t-1>2$, as by hypothesis $n\ge 3+(d-1)t>3+(d-2)t$. Moreover $Q$ is an equigenerated initial $t$--spread lexsegment ideal in $K[x_2,\dots,x_n]$, so it has a linear resolution. Hence, by (\ref{eq:pd(I)j1=2BettiSplit}),
	$$
	\pd(I)\ =\ \max\{\pd(P),\pd(Q),\pd(P\cap Q)+1\}.
	$$
	We have $\pd(P)\le1$. Indeed, $P$ is isomorphic to the final $t$--spread lexsegment ideal $(\mathcal{L}_t(u/x_1,x_{n-(d-2)t}\cdots x_{n-t}x_n))$ generated in degree $d-1$, and formula (\ref{Bettinumbersreversetspread}) gives
	$$
	\pd(I)=\begin{cases}
	n-(n-(d-2)t-1)-(d-2)t=1&\textup{if}\ u>u_{d-1},\\
	n-(n-(d-2)t)-(d-2)t=0&\textup{if}\ u=u_{d-1}.
	\end{cases}
	$$
	Analogously, as $Q$ is an initial $t$--spread lexsegment ideal in $K[x_2,\dots,x_n]$, $\pd(Q)=0$ if $v=v_1$, or $\pd(Q)=1$ if $v=v_2$. Thus, (\ref{eq:pd(I)j1=2BettiSplit}) implies that $\pd(I)=1$ if and only if $\pd(P\cap Q)=0$, \emph{i.e.}, if and only if $P\cap Q$ is a principal ideal.
\end{proof}

The next result points out that for $n\gg0$, precisely $n\ge 4+(2d-3)t$, the Cohen--Macaulay $t$--spread lexsegment ideals with $j_1=2$ are just complete intersection ideals.
\begin{Thm} \label{thm:unmixed}
	Let $u=x_{i_1}x_{i_2}\dots x_{i_d}$ and $v=x_{j_1}x_{j_2}\cdots x_{j_d}$ be $t$--spread monomials of degree $d\ge 2$ of $S$ with $i_1=1$, $j_1=2$, and let $I$ be the $t$--spread lexsegment ideal $I=(\mathcal{L}_t(u,v))$  of $S$. Assume $n\ge 4+(2d-3)t$, then the following conditions are equivalent:
	\begin{enumerate}
		\item[\textup{(a)}] $I$ is Cohen--Macaulay;
		\item[\textup{(b)}] $u=x_{1}x_{n-(d-2)t}\cdots x_{n-t}x_n$, $v=x_{2}x_{2+t}\cdots x_{2+(d-1)t}$.
	\end{enumerate}
\end{Thm}
\begin{proof}
	Let $\Delta$ be the simplicial complex on vertex set $[n]$ associated to $I$.
	\smallskip\\
	(a)$\Longrightarrow$(b). If $I=I_\Delta$ is Cohen--Macaulay, then $\Delta$ is pure. The set $[n]\setminus\{1,2\}$ is a facet of $\Delta$. Indeed, any monomial $w\in I$ has $1\in\supp(w)$ or $2\in\supp(w)$. Therefore, $\Delta$ pure implies $\dim(\Delta)=n-3$.
	Consider the set $G=[n]\setminus\{2,i_2,i_2+1,\dots,n-(d-2)t\}$. Observe that $G\in\Delta$, indeed $\p_{[n]\setminus G}=(x_2,x_{i_2},x_{i_2+1},\dots,x_{n-(d-2)t})$ is a monomial prime ideal that contains $I$. Now, the condition $\Delta$ is pure implies that there exists a minimal prime $\p\in\Min(I)$ of height two such that $\p\subseteq\p_{[n]\setminus G}$. The $t$--spread monomial $w=x_2x_{2+t}\cdots x_{2+(d-1)t}\in I$ as $u>w\ge v$. We have $x_2\in\p$. In fact, if $x_2\notin\p$, then $\p\subseteq\p_{[i_2,n-(d-2)t]}$. We must have $\p=\p_{[i_2,n-(d-2)t]}$. Indeed, if we omit some $x_j$, $j\in[i_2,n-(d-2)t]$, then we can find a monomial $z\in G(I)$, with $\min(z)=1$ and $\min(z/x_1)=j$, and this monomial does not belong to $\p$, a contradiction. $\Delta$ pure implies $\alt(\p)=2$, thus $i_2=n-(d-2)t-1$. But $[n-(d-2)t-1,n-(d-2)t]\cap\supp(w)=\emptyset$. Indeed, $n\ge 4+(2d-3)t$, so
	\begin{equation}\label{eq:calc4+(2d-3)t}
	\begin{aligned}
	n-(d-2)t-2\ &\ge\ 4+(2d-3)t-(d-2)t-2=2+(d-1)t\\
	&=\ \max(w).
	\end{aligned}
	\end{equation}
	Thus $\max(w)<n-(d-2)t-1$. This implies that $w\in I\setminus\p$, a contradiction. Hence, $x_2\in\p$. Therefore, $\alt(\p)=2$ and $\p=(x_2,x_{i_2})$. The monomial $z=x_{1}x_{n-(d-2)t}\cdots x_{n-t}x_n$ belongs to $I$. We have $z\in\p$ if and only if $x_{i_2}$ divides $z$. Thus, as $1+t\le i_2\le n-(d-2)t$, we have $i_2=n-(d-2)t$ and $u=x_{1}x_{n-(d-2)t}\cdots x_{n-t}x_n$.\\
	Let $r=\max\{\max(w):w\in G(I)\,\, \mbox{and $\min(w)=2$}\}$. Observe that $r\ge 2+(d-1)t$, as $x_2x_{2+t}\cdots x_{2+(d-1)t}\in I$. The ideal $\q=(x_1,x_{2+(d-1)t},x_{3+(d-1)t},\dots,x_{r})$ is a minimal prime of $I$. The condition $\Delta$ pure implies $\alt(\q)=2$ and $r=2+(d-1)t$. Thus, $\max(v)=2+(d-1)t$ and $v=x_{2}x_{2+t}\cdots x_{2+(d-1)t}$.
	\smallskip\\
	(b)$\Longrightarrow$(a). Let $u=x_1x_{n-(d-2)t}\cdots x_{n-t}x_{n}$ and $v=x_2x_{2+t}\cdots x_{2+(d-1)t}$.  By (\ref{eq:calc4+(2d-3)t}), we have $\supp(u)\cap\supp(v)=\emptyset$. Moreover $I=(u,v)$. Thus, $I$ is a complete intersection and so it is Cohen--Macaulay.
\end{proof}
\begin{Rem}
 \em One can observe that under the hypotheses of Theorem \ref{thm:unmixed}, the classification for Cohen--Macaulay $t$--spread lexsegment ideal is equivalent to the classification for unmixed $t$--spread lexsegment ideal (see also, Remark \ref{rem:unmixed}).
\end{Rem}

\subsection{The general case}

In this subsection,  we classify all Cohen--Macaulay $t$--spread lexsegment ideals $I=(\mathcal{L}_t(u,v))$ of $S= K[x_1, \ldots, x_n]$ with $\min(v)>2$. 

\subsubsection{The underlying idea behind the classification}

Let $I=(\mathcal{L}_t(u,v))$  be a Cohen--Macaulay $t$--spread lexsegment ideal of $S$ with $u=x_{i_1}x_{i_2}\dots x_{i_d}$, $v=x_{j_1}x_{j_2}\cdots x_{j_d}\in M_{n, d, t}$ such that
$\min(u)=i_1=1$, $\min(v)=j_2>2$, $d\ge2$. If  $\Delta$ is the simplicial complex on the vertex set $[n]$ associated to $I$,  then $\Delta$ is pure. This is a key property in order to get the classification. 
First, note that since $\min(u)=1$, then  
$$u\in\mathcal{L}_t(x_1x_{1+t}\cdots x_{1+(d-1)t},x_1x_{n-(d-2)t}x_{n-(d-3)t}\cdots x_{n-t}x_{n}).$$

Assume $u=x_1x_{1+t}\cdots x_{1+(d-1)t} = \max(M_{n, d, t})$, then $I=(\mathcal{L}_t(u,v))=(\mathcal{L}_t^i(v))$ is an initial $t$--spread lexsegment ideal. Hence, $\Delta$ is Cohen--Macaulay if and only if $I=I_{n,d,t}$ (Corollary \ref{cor:JCohenMac}). \\

Now, let $u<x_1x_{1+t}\cdots x_{1+(d-1)t}$. Then, $F=[1+(d-1)t]\in\Delta$. Moreover, $[1+(d-1)t]\cup\{k\}\notin\Delta$, for all $k\notin[1+(d-1)t]$. In fact, $j_1>2$ implies that $w=x_2x_{2+t}\cdots x_{2+(d-2)t}x_k\in I=I_\Delta$, thus $G=\supp(w)\notin\Delta$ and $G\subseteq F\cup\{k\}$. Therefore, $F$ is a facet of $\Delta$ and by the purity of $\Delta$ we have that
\begin{equation}\label{eq:dim(D)tLex(d-1)t}
\dim(\Delta)=|F|-1=(d-1)t.
\end{equation}
We show that $v\le x_{n-1-(d-1)t}x_{n-1-(d-2)t}\cdots x_{n-1-t}x_{n-1}$. Suppose on the contrary that $v> x_{n-1-(d-1)t}x_{n-1-(d-2)t}\cdots x_{n-1-t}x_{n-1}$, then $A=[n-1-(d-1)t,n]$ is a face of $\Delta$ and $\dim(\Delta)\ge\dim A=|A|-1=(d-1)t+1>(d-1)t=\dim(\Delta)$, a contradiction. Thus, $v\le x_{n-1-(d-1)t}x_{n-1-(d-2)t}\cdots x_{n-1-t}x_{n-1}$. Hence, 
$$v \in \{v_0>v_1>\dots>v_{d-1}>v_d=\min(M_{n,d,t})\},$$
where
\begin{equation}\label{eq:v_ellmons}
v_\ell=\Big(\prod_{s=\ell}^{d-1}x_{n-st-1}\Big)\Big(\prod_{s=0}^{\ell-1}x_{n-st}\Big), \quad \ell\in\{0,\dots,d\}.
\end{equation}
We can observe that $v\ne v_d$. Otherwise, $I=(\mathcal{L}_t(u,v))=(\mathcal{L}_t^f(u))$ is a final $t$--spread lexsegment ideal and so $I$ is Cohen--Macaulay if and only if $I$ is the $t$--spread Veronese ideal $I_{n,d,t}$ (Corollary \ref{cor:TCohenMac}). Thus $u=\max(M_{n,d,t})$, against the assumption that $u<\max(M_{n,d,t})$.\\
Finally, for $v$ there do exist three admissible choices: 
\begin{enumerate}
	\item[-] $v=v_{d-1}$, 
	\item[-] $t=1$ and $v=v_\ell$, for $\ell\in\{0,\dots,d-2\}$, 
	\item[-] $t>1$ and $v=v_\ell$, for $\ell\in\{0,\dots,d-2\}$.
\end{enumerate}

In particular, if $v=v_\ell=\big(\textstyle\prod_{s=\ell}^{d-1}x_{n-st-1}\big)\big(\prod_{s=0}^{\ell-1}x_{n-st}\big)$ ($\ell\in\{0,\dots,d-2\}$), then 
\begin{equation}\label{eq:importset2}
H=[n-(d-1)t-1,n]\setminus\{n-\ell t-1\}\in \Delta.
\end{equation}
Moreover, $|H|=1+(d-1)t$ and $\Delta$ pure implies that $H\cup\{1\}\notin\Delta$, \emph{i.e.}, ${\bf x}_{H\cup\{1\}}\in I=I_\Delta$.

\subsubsection{The classification}
Now we are in position to prove the main result in the article.
\begin{Thm}\label{Teor:ItLexCM>2}
	Let $u=x_{i_1}x_{i_2}\dots x_{i_d}$, $v=x_{j_1}x_{j_2}\cdots x_{j_d}$ be $t$--spread monomials of degree $d\ge 2$ of $S$ 
	with $i_1=1$, $j_1>2$ and let $I$ be the $t$--spread lexsegment ideal $I=(\mathcal{L}_t(u,v))$ of $S$. Assume $I\ne I_{n,d,t}$. Then $I$ is Cohen--Macaulay if and only if one of the following conditions holds true
	\begin{enumerate}
		\item[\textup{(a)}] $u=x_1x_{n-(d-2)t}\cdots x_{n-t}x_n$ and $v=x_{n-(d-1)t-1}\big(\textstyle\prod_{s=0}^{d-2}x_{n-st}\big)$;
		\item[\textup{(b)}] $t=1$ and for some integer $\ell\in\{0,\dots,d-2\}$,
		\begin{align*}
		u&=x_1\big(\textstyle\prod_{s=\ell+1}^{d-2}x_{n-s-1}\big)\big(\prod_{s=0}^{\ell}x_{n-s}\big),\ \ \ \ &v=\big(\textstyle\prod_{s=\ell}^{d-1}x_{n-s-1}\big)\big(\prod_{s=0}^{\ell-1}x_{n-s}\big);
		\end{align*}
		\item[\textup{(c)}] $t>1$ and for some integer $\ell\in\{0,\dots,d-2\}$,
		\begin{align*}
		u&\in\L_t\big(x_1\big(\textstyle\prod_{s=\ell}^{d-2}x_{n-st-1}\big)\big(\prod_{s=0}^{\ell-1}x_{n-st}\big),x_1\big(\prod_{s=0}^{d-2}x_{n-st}\big)\big),\\
		v&=\big(\textstyle\prod_{s=\ell}^{d-1}x_{n-st-1}\big)\big(\prod_{s=0}^{\ell-1}x_{n-st}\big).
		\end{align*}
	\end{enumerate}
\end{Thm}
\begin{proof} (a) Let $v=v_{d-1}=x_{n-(d-1)t-1}\big(\prod_{s=0}^{d-2}x_{n-st}\big)$. 
	The only monomial $w\in M_{n,d,t}$ with $w< v$ is $w=\min(M_{n,d,t})=\big(\textstyle\prod_{s=0}^{d-1}x_{n-st}\big)$. Since $x_1(w/x_{n-(d-1)t})\le u$, then $I$ is a completely $t$--spread lexsegment ideal (Theorem \ref{thm:compltspreadlex}(2)). Moreover, $I$ has a linear resolution (Theorem \ref{thm:compltspreadlex}(2)). Indeed, setting $u=x_{1}x_{i_2}\cdots x_{i_d}$, we have $i_2>1+t$ and 
	$$
	x_1x_{i_2-t}x_{i_3-t}\cdots x_{i_d-t}\ge x_1x_{n-(d-1)t}x_{n-(d-2)t}\cdots x_{n-t}=x_1(w/x_{\max(w)}).
	$$

	Note that the Cohen--Macaulayness of $\Delta$ implies $\dim(S/I)=\depth(S/I)$. From (\ref{eq:dim(D)tLex(d-1)t}), $\dim(S/I)=n-\height(I)=1+(d-1)t = \depth(S/I)$.

	On the other hand, by the Auslander--Buchsbaum formula, $\depth(S/I)=n-\pd(S/I)=n-1-\pd(I)$. Thus, $S/I$ is Cohen--Macaulay if and only if $\pd(I)=n-(d-1)t-2$. 
	Setting $a=n-1-(d-1)t$ and $b=a+d$, we can note that $\pd(I)\le n-(d-1)t-1$ if and only if (Theorem \ref{thm:compltspreadlex}(3)) 
	\begin{align*}
	\beta_{a}(I)= \beta_{a,b}(I)&=\sum_{w\in\mathcal{L}_t^f(u)}\binom{n-\min(w)-(d-1)t}{n-1-(d-1)t}-\sum_{\substack{w\in\mathcal{L}_t^f(v)\\ w\ne v}}\binom{\max(w)-(d-1)t-1}{n-1-(d-1)t}\\
	&=\big|\big\{ w\in\mathcal{L}_t^f(u):\min(w)=1\big\}\big|-\big|\big\{w\in\mathcal{L}_t^f(v)\setminus\{v\}:\max(w)=n\big\}\big|\\
	&=\big|\big\{ w\in\mathcal{L}_t^f(u):\min(w)=1\big\}\big|-1 = 0,
	\end{align*}
	\emph{i.e.}, $u=x_1\big(\prod_{s=0}^{d-2}x_{n-st}\big)$. Condition (a) follows.
	
	Conversely, let $u=x_1x_{n-(d-2)t}\cdots x_{n-t}x_n$ and $v=x_{n-(d-1)t-1}\big(\textstyle\prod_{s=0}^{d-2}x_{n-st}\big)$. We prove that $\dim(S/I) =\depth(S/I)$.
	
	First, using the same arguments as before, we can verify that $I$ is a completely $t$--spread lexsegment ideal which has a linear resolution.
	
 Setting $a = n-(d-1)t-1$ and $c=a-1=n-(d-1)t-2$, we have $m=\big|\{w\in\mathcal{L}_t^f(u):\min(w)=2\}\big|>1$ and 
	\begin{align*}
	\beta_{c}(I)&=\sum_{w\in\mathcal{L}_t^f(u)}\binom{n-\min(w)-(d-1)t}{n-2-(d-1)t}-\sum_{\substack{w\in\mathcal{L}_t^f(v)\\ w\ne v}}\binom{\max(w)-(d-1)t-1}{n-2-(d-1)t}
	\\
	&=\binom{a}{a-1}+m\binom{a-1}{a-1}-\binom{a}{a-1}=m>0.
	\end{align*}
	by  Theorem \ref{thm:compltspreadlex}(3). Thus, $\pd(I)=c=n-(d-1)t-2$. Hence, from the Auslander--Buchsbaum formula, $\depth(S/I)= 1+(d-1)t$.
	
	Let us prove that  $\dim(S/I) = 1+(d-1)t$.  It is sufficient to verify that $\dim(\Delta)=(d-1)t$.
	Using the notation of Theorem \ref{primdecompgeneraltspreadlex}, since $I$ is a completely $t$--spread lexsegment ideal, we have $\mathcal{F}(\Delta)=\mathcal{G}\cup\{[n]\setminus F_p:p\in\mathcal{I}\}\cup\widetilde{\mathcal{F}}$. 
	If $F\in\mathcal{G}$, then $|F|=1+(d-1)t$. 
	If $F\in\widetilde{\mathcal{F}}$, then $|F|=(d-1)t<1+(d-1)t$. Finally, we check that $\mathcal{I}=\{1\}$ and $|[n]\setminus F_1|=1+(d-1)t$. Indeed, for $p=1$, $[n]\setminus F_1=[j_1+1,n]=[n-(d-1)t,n]$ has cardinality $1+(d-1)t$, so $1\in\mathcal{I}$. Let $2\le p\le d$, then
	\begin{align*}
	[n]\setminus F_p&=\bigcup_{i=1}^{p-1}[j_i,j_i+(t-1)]\cup[j_p+1,n]\\
	&=\bigcup_{i=1}^{p-1}[j_i,j_i+(t-1)]\cup[n-(d-p)t+1,n]
	\end{align*}
	has cardinality $(p-1)t+n+1-(n-(d-p)t+1)=(d-1)t$ and moreover,
	$$
	\min(v/x_{j_p})=\min(v)=n-(d-2)t-1<n-(d-2)t-1\le\min(u/x_1).
	$$
	Thus $v/x_{j_p}>u/x_1$ and $p\notin\mathcal{I}$ and so $\mathcal{I}=\{1\}$ and $|[n]\setminus F_1|=1+(d-1)t$. Hence, $\dim(\Delta)=(d-1)t$ and consequently $\dim(S/I)=n-\height(I)=1+(d-1)t = \depth(S/I)$ and $I$ is Cohen--Macaulay.\\\\
(b) Let $t=1$. Let us consider the set $H$ defined in (\ref{eq:importset2}). We note that the smallest $1$--spread (squarefree) monomial of degree $d$ that we can ``extract" from $H\cup\{1\}=\{1,n-d,n-d+1,\dots,n-\ell-2,n-\ell,\dots,n-1,n\}$ is
	\begin{align*}
	z&=x_1x_{n-d+1}\cdots x_{n-\ell-2}\cdot x_{n-\ell}\cdots x_{n-1}x_n\\
	&=x_1\big(\textstyle\prod_{s=\ell+1}^{d-2}x_{n-s-1}\big)\big(\prod_{s=0}^{\ell}x_{n-s}\big).
	\end{align*}
	Thus, we must have $z\in I$, and so $u\ge z$, \emph{i.e.},
	$$
	u\in\L_t\big(x_1\big(\textstyle\prod_{s=\ell}^{d-2}x_{n-s-2}\big)\big(\prod_{s=0}^{\ell-1}x_{n-s-1}\big),x_1\big(\prod_{s=\ell}^{d-2}x_{n-s-1}\big)\big(\prod_{s=0}^{\ell+1}x_{n-s}\big)\big).
	$$
	Let $v=v_\ell$, for some $\ell\in \{0, \ldots, d-2\}$. The monomials $w\in M_{n,d,1}$ with $w< v$ are $v_{\ell+1}>\dots>v_d$. Since, for $r=\ell+1,\dots,d$, $$x_1(v_r/x_{n-(d-1)})\le x_1\big(\textstyle\prod_{s=\ell+1}^{d-2}x_{n-s-1}\big)\big(\prod_{s=0}^{\ell}x_{n-s}\big)\le u,$$ $I$ is a completely $1$--spread lexsegment ideal. Moreover, $I$ has a linear resolution. Indeed, setting $u=x_{1}x_{i_2}\cdots x_{i_d}$, we have $i_2>1+t=2$ and
	$$
	x_1x_{i_2-1}x_{i_3-1}\cdots x_{i_d-1}\ge x_1\big(\textstyle\prod_{s=\ell+1}^{d-2}x_{n-s-2}\big)\big(\prod_{s=0}^{\ell}x_{n-s-1}\big)> x_1(v_{\ell+1}/x_{\max(w)}).
	$$
	From (\ref{eq:dim(D)tLex(d-1)t}), $\dim(S/I) = 1+(d-1)t = d$ and, since we are assuming $I$ Cohen--Macaulay, then $\depth(I) =d$. 
	On the other hand from the Auslander--Buchsbaum formula, $S/I$ is Cohen--Macaulay if and only if $\pd(I)=n-d-1$.  
	Setting $a=n-d$ and $b=a+d$, 
	$\pd(I)\le n-d$, 
	$S/I$ is Cohen--Macaulay if and only if
	\begin{align*}
	\beta_{a}(I)=\beta_{a,b}(I)&=\big|\big\{ w\in\mathcal{L}_t^f(u):\min(u)=1\big\}\big|-\big|\big\{w\in\mathcal{L}_t^f(v)\setminus\{v\}:\max(w)=n\big\}\big|\\
		&=\big|\big\{ w\in\mathcal{L}_t^f(u):\min(w)=1\big\}\big|-\big|\big\{v_{\ell+1}>\dots>v_d\big\}\big|\\
	&=\big|\big\{ w\in\mathcal{L}_t^f(u):\min(w)=1\big\}\big|-(d-\ell)=0,
	\end{align*}
	\emph{i.e.}, $u=x_1\big(\textstyle\prod_{s=\ell+1}^{d-2}x_{n-s-1}\big)\big(\prod_{s=0}^{\ell}x_{n-s}\big)$. The assertion (b) follows.
	
	Conversely, let $u=x_1\big(\textstyle\prod_{s=\ell+1}^{d-2}x_{n-s-1}\big)\big(\prod_{s=0}^{\ell}x_{n-s}\big)$ and $v=v_\ell$, for some $\ell$ $\in\{0,\dots,d-2\}$. 
	We compute $\dim(S/I)$ and $\depth(S/I)$. 
	
	Using the notation of Theorem \ref{primdecompgeneraltspreadlex}, $\mathcal{F}(\Delta)=\mathcal{G}\cup\{[n]\setminus F_p:p\in\mathcal{I}\}\cup\widetilde{\mathcal{F}}$. We show that $\dim(\Delta)= d$. 
	If $F\in\mathcal{G}$, then $|F|=d$. 
	Arguing as in the proof of condition (b), one can observe that $\mathcal{I}=[d-\ell]$ and for $p\in\mathcal{I}$, $|[n]\setminus F_p|=d$. 
		Furthermore, for any $F\in{\mathcal{F}}$, $|F|=d-1<d$. Hence $\dim(\Delta)=d$ 
	and so $\dim(S/I)=d$.\\
	
	Setting $a=n-d$ 
	and $c=a-1=n-d-1$, 
	we have $\big|\{w\in \mathcal{L}_t^f(u):\min(w)=1\}\big|=d-\ell$, $m=\big|\{w\in\mathcal{L}_t^f(u):\min(w)=2\}\big|>1$ and
	\begin{align*}
	\beta_{c}(I)&=\sum_{w\in\mathcal{L}_t^f(u)}\binom{n-\min(w)- d+1}{n-d-1}-\sum_{\substack{w\in\mathcal{L}_t^f(v)\\ w\ne v}}\binom{\max(w)-d}{n-d-1}
	\\
	&=(d-\ell)\binom{a}{a-1}+m\binom{a-1}{a-1}-(d-\ell)\binom{a}{a-1}=m>0.
	\end{align*}
	Thus, $\pd(I)=n-d-1$ 
	and from the Auslander--Buchsbaum formula $\depth(I)=d$ and $I$ is Cohen--Macaulay. 
\end{proof}

For the proof of the last case (c), we need some more preparations.

First, recall that if $0\rightarrow L\rightarrow M\rightarrow N\rightarrow0$ is  a short exact sequence of finitely generated graded $S$--modules, then $\pd(M)\le\max\{\pd(L),\pd(N)\}$. Moreover, given a monomial ideal $I$ of $S$ and a monomial $u$ of $S$, there exists the short exact sequence
$$
0\rightarrow S/(I:u)\rightarrow S/I\rightarrow S/(I,u)\rightarrow0.
$$
Therefore  $\pd(S/I)\le\max\{\pd(S/(I:u)),\pd(S/(I,u))\}$.\smallskip

\begin{Lem}\label{lem:pdJ<I}
	Let $I,J$ equigenerated monomial ideals of $S$ having initial degree $d$. Suppose $J\subseteq I$, then
	$$
	\beta_{i,i+d}(J)\le\beta_{i,i+d}(I),
	$$
	for all $i\ge0$. Moreover, if $J$ has a linear resolution, $\pd(I)\ge\pd(J)$.
\end{Lem}
\begin{proof} The first assertion follows from \cite[Lemma 10.3.5]{JT}. For the second assertion, assume $J$ has a linear resolution, and let $\pd(J)=s$. Then $\beta_s(J)=\beta_{s,s+d}(J)\le\beta_{s,s+d}(I)$. Finally, $\pd(I)\ge s=\pd(J)$.
\end{proof}

Now we are in position to prove the last statement of Theorem \ref{Teor:ItLexCM>2}.\\\\
(c) Let $t\ge2$. From (\ref{eq:importset2}), the smallest $t$--spread monomial we can ``extract" from $H\cup\{1\}$ is
	$$
	z=x_1x_{n-(d-2)t}x_{n-(d-3)t}\cdots x_{n-t}x_n.
	$$
	Since $z$ is the smallest $t$--spread monomial of $S$ with minimum equal to $1$, then $z\in I$, and there are no further conditions to impose on the monomial $u$. Thus
	\begin{equation}\label{eq:Setcase(d)}
	u\in\L_t\big(x_1\big(\textstyle\prod_{s=\ell}^{d-2}x_{n-st-2}\big)\big(\prod_{s=0}^{\ell-1}x_{n-st-1}\big),x_1\big(\prod_{s=0}^{d-2}x_{n-st}\big)\big).
	\end{equation}
	
	Setting 	$$
	u_p=x_1\Big(\prod_{s=p}^{d-2}x_{n-st-1}\Big)\Big(\prod_{s=0}^{p-1}x_{n-st}\Big), \,\,\, p=\ell,\dots,d-1,
	$$
	we have that $$\L_t\big(x_1\big(\textstyle\prod_{s=\ell}^{d-2}x_{n-st-1}\big)\big(\prod_{s=0}^{\ell-1}x_{n-st}\big),x_1\big(\prod_{s=0}^{d-2}x_{n-st}\big)\big)=\big\{u_\ell>u_{\ell+1}>\dots>u_{d-1}\big\}.$$
	We prove that if 
	\begin{equation}\label{eq:noCM}
	u >u_\ell=x_1\big(\textstyle\prod_{s=\ell}^{d-2}x_{n-st-1}\big)\big(\prod_{s=0}^{\ell-1}x_{n-st}\big), 
	\end{equation}
	then $I$ is not Cohen--Macaulay.
	
	By hypothesis $v=v_\ell$ ($\ell =0, \ldots, d-2$) and from (\ref{eq:noCM}) 
	$$
	u\ge x_1\big(\textstyle\prod_{s=\ell-1}^{d-2}x_{n-st-1}\big)\big(\prod_{s=0}^{\ell-2}x_{n-st}\big),
	$$
	where $x_1\big(\textstyle\prod_{s=\ell-1}^{d-2}x_{n-st-1}\big)\big(\prod_{s=0}^{\ell-2}x_{n-st}\big)$ is the greatest monomial with minimum equal to $1$ preceding $u_\ell$, with respect to the squarefree lex order. The monomials $w\in M_{n,d,t}$ smaller than $v$ are $v_{\ell+1}>\dots>v_d$ and for each $v_q$, $q=\ell+1, \ldots, d$, we have
	$$
	u\ge x_1\big(\textstyle\prod_{s=\ell-1}^{d-2}x_{n-st-1}\big)\big(\prod_{s=0}^{\ell-2}x_{n-st}\big)=x_1(v_{\ell+1}/x_{\min(v_{\ell+1})})>x_1(v_q/x_{\min(v_q)}).
	$$
	Hence, setting $u=x_1x_{i_2}\cdots x_{i_d}$, we have $i_2>1+t$ and
	\begin{align*}
	x_1x_{i_2-t}\cdots x_{i_d-t}&\ge x_1x_{n-(d-1)t-1}\cdots x_{n-\ell t-1}\cdot x_{n-(\ell-1)t}\cdots x_{n-t}\\
	&=x_1(v_{\ell}/x_{\max(v_\ell)})\ge x_1(v_{\ell+1}/x_{\max(v_{\ell+1})}).
	\end{align*}
	Therefore, $I$ is a completely $t$--spread lexsegment ideal with a linear resolution. Setting $a=n-1-(d-1)t$ and $b=a+d$, we have
	\begin{align*}
	\beta_{a}(I)=\beta_{a,b}(I)&=\big|\big\{ w\in\mathcal{L}_t^f(u):\min(u)=1\big\}\big|-\big|\big\{w\in\mathcal{L}_t^f(v)\setminus\{v\}:\max(w)=n\big\}\big|\\
	&\ge\big|\big\{ u>u_{\ell}>\dots>u_{d-1}\}\big|-(d-\ell)= d+1-\ell-(d-\ell)= 1.
	\end{align*}
	It follows that $\pd(I)=n-1-(d-1)t$ and consequently $I$ is not Cohen--Macaulay. Condition (c) follows.
	
	Conversely, assume (c) holds. We prove that $\dim(S/I)=1+(d-1)t$. This is equivalent to verify that $\dim(\Delta)=(d-1)t$. For this aim it sufficies to show the following facts:
	\begin{enumerate}
		\item[(i)] there exists $F\in\mathcal{F}(\Delta)$ with cardinality $|F|=1+(d-1)t$;
		\item[(ii)] for any $A\subseteq[n]$ with $|A|=2+(d-1)t$, then $A\notin\Delta$.
	\end{enumerate}
	\noindent
	(i) Let us consider $F=[n-(d-1)t,n]$. By the structure of $v$, it follows that $F\in\Delta$ and furthermore $|F|=1+(d-1)t$, as desired.\medskip
	\\ 
	(ii) Let $A\subseteq[n]$ with $|A|=2+(d-1)t$. We have $1\le\min(A)\le n-(d-1)t-1$.
	
	If $\min(A)=n-(d-1)t-1$, then $A=[n-(d-1)t-1,n]\notin\Delta$, as
	$$
	x_{n-(d-1)t-1}x_{n-(d-2)t-1}\cdots x_{n-t-1}x_{n-1}\in I.
	$$
	
	If $1<\min(A)<n-(d-1)t-1$, setting $\ell_1=\min(A)$ and $\ell_r=\min\{\ell\in A:\ell\ge\ell_{r-1}+t\}$, for $r\ge 2$, the sequence $L:\ell_1>\ell_2>\cdots$ has at least $d$ terms, and ${\bf x}_L\in I$. Thus, $A\notin\Delta$.
	
	If $\min(A)=1$, we consider the sequence $L$ above defined. If ${\bf x}_L\le u$, then $A\notin\Delta$. Otherwise, set $B=A\setminus\{1\}$. 
	
	If $\min(B)=n-(d-1)t$, then $A=\{1\}\cup B=\{1\}\cup[n-(d-1)t,n]\notin\Delta$, as $x_1x_{n-(d-2)t}\cdots x_{n-t}x_{n}\in I$. If $\min(B)<n-(d-1)t-1$, setting $\nu_1=\min(B)$, and $\nu_r=\min\{\nu\in B:\nu\ge\nu_{r-1}+t\}$, for $r\ge 2$, the sequence $N:\nu_1>\nu_2>\cdots$ has at least $d$ terms, and ${\bf x}_N\in I$, implying that $A\notin\Delta$. 
	
	Let $\min(B)=n-(d-1)t-1$. Then $A=\{1\}\cup B=\{1\}\cup[n-(d-1)t-1,n]\setminus\{m\}$, for some $m>n-(d-1)t-1$.
	We can write $m=n-pt-q$, with $0\le p\le d-3$ and $0\le q\le t-1$. If $q\ge2$, then $\supp(v)\subseteq B\subseteq A$, and $A\notin\Delta$. If $q=1$, then $\{1,n-(d-2)t,\dots,n-t,n\}\subseteq A$, and $x_1x_{n-(d-2)t}\cdots x_{n-t}x_n\in I$, thus $A\notin\Delta$. If $q=0$, we consider $p$. If $p<\ell-1$, the structure of $v$ implies that $\supp(v)\subseteq A$, and once again $A\notin\Delta$. If $p\ge\ell-1$, then $C=\{n-kt-1:k=0,\dots,d-1\}\subseteq A$, and ${\bf x}_C>v$, so ${\bf x}_C\in I$, implying $A\notin\Delta$. Finally, in all possible cases, $A\notin\Delta$, and (ii) follows.
	
	Finally, $\dim(S/I)=1+(d-1)t$.
	
	Now, set $I_{u_p}=(\mathcal{L}_t(u_p,v))$, $p=\ell,\dots,d-1$. 
	We prove that $I_{u_{p}}$ is Cohen--Macaulay, 
	\emph{i.e.}, $\pd(I_{u_p})=n-2-(d-1)t$. 
	
	First, assume $p=\ell$. We have 
	$$
	u=u_\ell=x_1(v_\ell/x_{\min(v_\ell)})>x_1(v_q/x_{\min(v_q)}),\,\,\, q=\ell+1,\dots,d
	$$
	and setting 
	$u=x_1x_{i_2}\cdots x_{i_d}$, it follows that
	$$
	x_1x_{i_2-t}\cdots x_{i_d-t}=x_1(v_{\ell}/x_{\max(v_\ell)}).
	$$
	Thus, $I_{u_\ell}$ is a completely $t$--spread lexsegment ideal with a linear resolution. Set $a=n-1-(d-1)t$, $b=a+d$, we have
	\begin{align*}
	\beta_{a}(I_{u_\ell})=\beta_{a,b}(I_{u_\ell})&=\big|\big\{ w\in\mathcal{L}_t^f(u):\min(u)=1\big\}\big|-\big|\big\{w\in\mathcal{L}_t^f(v)\setminus\{v\}:\max(w)=n\big\}\big|\\
	&=\big|\big\{u_\ell>u_{\ell+1}>\dots >u_{d-1}\}\big|-|\{v_{\ell+1}>\dots>v_d\}|=0,
	\end{align*}
	and, as $m=\big|\{w\in\mathcal{L}_t^f(u):\min(w)=2\}\big|>1$, then
	\begin{align*}
	\beta_{a-1}(I_{u_\ell})&=\sum_{w\in\mathcal{L}_t^f(u)}\binom{n-\min(w)-(d-1)t}{n-2-(d-1)t}-\sum_{\substack{w\in\mathcal{L}_t^f(v)\\ w\ne v}}\binom{\max(w)-(d-1)t-1}{n-2-(d-1)t}
	\\
	&=(d-\ell)\binom{a}{a-1}+m\binom{a-1}{a-1}-(d-\ell)\binom{a}{a-1}=m>0.
	\end{align*}
	Thus, $\pd(I_{u_\ell})=a-1=n-(d-1)t-2$, and $I_{u_\ell}$ is Cohen--Macaulay, as desired.
	
	Now let $p=\ell+1$. Firstly, observe that $J=(\L_t(x_2x_{2+t}\cdots x_{2+(d-1)t},v))\subseteq I_{u_{\ell+1}}$. Moreover, $J$ is an initial $t$--spread lexsegment ideal in $K[x_2,\dots,x_{n}]$ with $\pd(J)=n-2-(d-1)t$. Thus $J$ has a linear resolution, and Lemma \ref{lem:pdJ<I} implies $\pd(I_{u_{\ell+1}})\ge n-2-(d-1)t$. For the other inequality, consider the short exact sequence:
	$$
	0\rightarrow S/(I_{u_{\ell+1}}:u_{\ell})\rightarrow S/I_{u_{\ell+1}}\rightarrow S/(I_{u_{\ell+1}},u_\ell)\rightarrow0.
	$$
	Observe that $S/(I_{u_{\ell+1}},u_\ell)=S/I_{u_{\ell}}$, and $\pd(S/I_{u_{\ell}})=\pd(I_{u_{\ell}})+1=n-1-(d-1)t$. 
	Let us verify that $I_{u_{\ell+1}}:u_\ell=\mathfrak{p}_{[2,n-(d-1)t-1]\cup\{n-\ell t\}}$ is a complete intersection and $\pd(S/(I_{u_{\ell+1}}:u_\ell))=n-1-(d-1)t$.
		
	By \cite[Proposition 1.2.2]{JT}, a set of generators for $I_{u_{\ell+1}}:u_\ell$ is given by $$\big\{w/\gcd(w,u_\ell):w\in G(I_{u_{\ell+1}})\big\}.$$
	
	If $w=u_p$, $p=\ell+1,\dots,d-1$, we have $u_{p}/\gcd(u_{p},u_\ell)=x_{n-(p-1)t}\cdots x_{n-\ell t}$ and for $p=\ell+1$, $u_{\ell+1}/\gcd(u_{\ell+1},u_\ell)=x_{n-\ell t}$, thus $x_{n-\ell t}$ divides all these generators.
	
	If $w\in\mathcal{L}_t(x_2x_{2+t}\cdots x_{2+(d-1)t},v)$, then $\min(w)\in[2,n-(d-1)t-1]$. Note that $[2,n-(d-1)t-1]\cap\supp(u_\ell)=\emptyset$, thus $x_{\min(w)}$ divides $w/\gcd(w,u_\ell)$ and for $z=x_{\min(w)}(u_\ell/x_{1})\in G(I_{u_{\ell+1}})$, we have that $z/\gcd(z,u_\ell)=x_{\min(w)}\in I_{u_{\ell+1}}:u_\ell$. 
	
	Finally, we have verified that $I_{u_{\ell+1}}:u_\ell=\mathfrak{p}_{[2,n-(d-1)t-1]\cup\{n-\ell t\}}$ and consequently $\pd(S/(I_{u_{\ell+1}}:u_\ell))=n-1-(d-1)t$.

	Hence,
	$$
	\pd(S/I_{u_{\ell+1}})\le\max\big\{ \pd(S/I_{u_{\ell}}), \pd(S/(I_{u_{\ell+1}}:u_\ell))\big\}=n-1-(d-1)t,
	$$
	so $\pd(I_{u_{\ell+1}})\le n-2-(d-1)t$. Thus $\pd(I_{u_{\ell+1}})=n-2-(d-1)t$ and $I_{u_{\ell+1}}$ is Cohen--Macaulay.
	For $p\in \{\ell+2,\dots,d\}$ the same arguments work and the proof of Theorem \ref{Teor:ItLexCM>2} is complete.

\section{Conclusions and Perspectives}\label{sec4}

In this article, we have investigated the Cohen--Macaulayness of $t$--spread lexsegment ideals using the theory of simplicial complexes. The completely $t$--spread lexsegment ideals have played an essential role. Non--completely $t$--spread lexsegment ideals are much less understood. Here are some possible questions to investigate.

\begin{Que}
	\rm Determine the standard primary decomposition of non--completely $t$--spread lexsegment ideals.
\end{Que}
\begin{Que}
	\rm Classify the pure simplicial complexes associated to $t$--spread lexsegment ideals.
\end{Que}

In \cite{OO}, Olteanu classified all \textit{sequentially Cohen--Macaulay} squarefree completely lexsegment ideals. To the best of our knowledge, a classification for squarefree non--completely lexsegment ideals is unknown. More generally, we ask the following
\begin{Que}
	\rm Classify all sequentially Cohen--Macaulay $t$--spread lexsegment ideals.
\end{Que}

Finally, due to many examples that we performed and our results on completely $t$--spread lexsegment ideals, we are led to conjecture the following:
\begin{Conj}\label{Conj:dimpdtspreadLex}
	\rm Let $I$ be a $t$--spread lexsegment ideal generated in degree $d\ge2$. Then $\dim(S/I)\ge(d-1)t$.
\end{Conj}

Note that Conjecture \ref{Conj:dimpdtspreadLex} is trivially true when $t=1$. Indeed, if $\Delta$ is the simplicial complex associated to $I$, then each set $A\subseteq[n]$, with $|A|=d-1$, is in $\Delta$. Thus $\dim(S/I)\ge d-1$. Conjecture \ref{Conj:dimpdtspreadLex} is also true for completely $t$--spread lexsegment ideals by virtue of Theorems \ref{primdecompinitialtspreadlex}, \ref{primdecompfinaltspreadlex}, \ref{primdecompgeneraltspreadlex}.\\\\
\emph{Acknowledgement}. We thank the the referee for his/her helpful suggestions that allowed us to improve the quality of the paper.

\end{document}